\documentclass{amsart}
\usepackage{amssymb}
\usepackage{amsmath}
\usepackage{graphicx}
\usepackage[english]{babel}
\usepackage[pdftitle={HMM Spin},
  pdfauthor={Doghonay, Runborg},
  pdffitwindow=true,
  breaklinks=true,
  colorlinks=true,
  urlcolor=blue,
  linkcolor=red,
  citecolor=red,
  anchorcolor=red]{hyperref}
\usepackage[numbers]{natbib}

\numberwithin{equation}{section}
\numberwithin{table}{section}
\numberwithin{figure}{section}

\usepackage{graphicx}
\usepackage{amsfonts}
\usepackage{amsmath}

\newcommand{\bm}{{\bf{m}}}
\newcommand{\bH}{{\bf{H}}}

\newcommand{\bM}{{\bf{M}}}
\newcommand{\bN}{{\bf{N}}}
\newcommand{\bF}{{\bf{F}}}
\newcommand{ \be}{{\bf{e}}}
\newcommand{\bCalH}{{\bf\mathcal{H}}}
\newcommand{\e}{\varepsilon}

\newtheorem{Remark}{Remark}
\newtheorem{Lemma}{Lemma}

\newtheorem{Theorem}{Theorem}

\begin{document}

\title[Multiscale Modeling in Micro Magnetism]{Temporal Upscaling in Micro Magnetism Via Heterogeneous Multiscale Methods}

\author[D.~Arjmand]{Doghonay Arjmand}
\author[M.~Poluektov]{Stefan Engblom}
\author[G.~Kreiss]{Gunilla Kreiss} 
\email{doghonay.arjmand@it.uu.se,stefane@it.uu.se,gunilla.kreiss@it.uu.se}

\address{Division of Scientific Computing \\
  Department of Information Technology \\
  Uppsala University \\
  SE-751 05 Uppsala, Sweden.}





\begin{abstract}
We consider a multiscale strategy addressing the disparate scales in the Landau-Lifschitz equations in micro-magnetism.  At the microscopic scale, the dynamics of magnetic moments are driven by  a high frequency field.  
On the macroscopic scale we are interested in simulating the dynamics of the magnetisation without fully resolving the microscopic scales.


  The method follows the framework of heterogeneous multiscale methods
  and it has two main ingredients: a micro- and a macroscale model.
  The microscopic model is assumed to be known exactly whereas the
  macro model is incomplete as it lacks effective quantities. The two models use different temporal and spatial
  scales and effective parameter values for the macro model are
  computed on the fly, allowing for improved efficiency over
  traditional one-scale schemes.
  
  For the analysis, we consider a single spin under a high frequency field
  and show that effective quantities can be obtained accurately
  with step-sizes much larger than the size of the microscopic scales
  required to resolve the microscopic features. Numerical results both for a single 
  magnetic particle as well as a chain of interacting magnetic particles are given to validate the 
  theory.

\end{abstract}

\maketitle

\section{Introduction}

Suppose that we are given an ensemble of particles  $\{ i \}_{i=1}^{N}$, each of them possessing a magnetic moment represented by ${\bm}_i(t) \in \mathbb{R}^{3}$, for all $t \in [0,T]$. At a microscopic level, the dynamics of the particle $i$ is modeled by the Landau-Lifschitz (LL) equation, \cite{Landau_Lifshits_1}:
\begin{equation} \label{LLG_eqn}
\begin{array}{lll}
\dfrac{d}{dt} \bm_i  = -\beta \bm_i \times   \bH_i -  \gamma \bm_i \times \left(  \bm_i \times  \bH_i \right), \quad i=1,\ldots,N.
\end{array}
\end{equation}
The first term on the right hand side accounts for the precessional
motion of the magnetisation $\bm_i$ around a field $\bH_i$, while the
nonlinear term is responsible for damping the magnetisation toward the
field $\bH_i$. In general, the effective field $\bH_i$ includes the effects of 
different short and long range interactions. The short range terms are due to exchange interactions, 
material anisotropy, and applied external field. The exchange interaction makes the neighbouring 
particles be aligned with each other. The name material anisotropy comes from the fact that when no external field is applied, the direction of the magnetic moments would be aligned in certain directions in the crystal lattice. The long range 
terms include magnetostatic, and the magnostrictive energies. The former accounts for the interaction of the magnetic moments over 
long distances, and the latter is related to the relation between the applied stress and change in the magnetisation, see \cite{Aharoni_Book,Cimrak} for more details.  In the presence of short range interactions, $\bH_i$ is given by
$$
\bH_i  =   \sum_{j} J_{ij} \bm_j  + K_{ani}  {\bf{p}} \cdot  \bm_i  + \tilde{\bH},
$$
where $J_{ij}$ is the exchange coefficient between the particles $i$ and $j$,  
${\bf{p}}$ is the material anisotropy which is the energetically
favourable direction of magnetisation, and $\tilde{\bH}$ is an
external field interacting with the particles. The long range interactions are ignored in the present study. 
The theoretical results in this paper will cover the contribution of the external field $\tilde{\bH}$ only, but 
the method itself will be extended and tested when short range interactions are also present.




We are often interested in the dynamics of the magnetisation at scales
much larger than the size of the spatial and temporal variations at a
particle level. It is, however, computationally unaffordable to solve
the microscopic model over the entire domain to simulate
\emph{macroscopic dynamics} of the magnetisation. The
macroscopic quantities can be defined as e.g.  local averages, in time
and space, of the magnetisations of the particles. In general, it is not possible to write explicit 
equations for these local averages unless certain restrictive assumptions/approximations are made. Such
an approach, however, lacks generality as the approximation may not be
valid e.g. in the vicinity of microscopic irregularities, such as defects. To treat
microscopic irregularities, it is therefore necessary to develop
multiscale strategies that can iterate back and forth between an
accurate microscopic and a computationally efficient macroscopic
model. The microscopic model is expensive and, therefore, should be
used only when necessary.

There are two common types of multiscale strategies to couple the
disparate scales in multiscale problems. The first approach is the
domain partitioning strategy which implies an explicit interface
between mathematical laws valid at different scales. The other
alternative is to use methods based on upscaling, where a macroscopic
model is assumed everywhere and the microscopic information is
upscaled to the macro model only in a small part of the domain. Meanwhile, the macro and the micro models exchange 
information in a consistent manner. Schematics of these approaches are
given in Figure \ref{Fig_MultiscaleApproaches}. There are a large
number of related literature on coupling different length scales in
magnetism. Without being exhaustive we refer the curious reader to
\cite{Banas_Prohl_etal_Book,Bruckner_etal,Chen_Cervera_Yang,Dobrovitski_Katnelson_Harmon,Grinstein_Koch_1,Jourdan_etal_2008,Poluektov_Eriksson_Kreiss,Tsiantos_etal}.

 \begin{figure}[h] 
    \centering
        \includegraphics[width=0.47\textwidth]{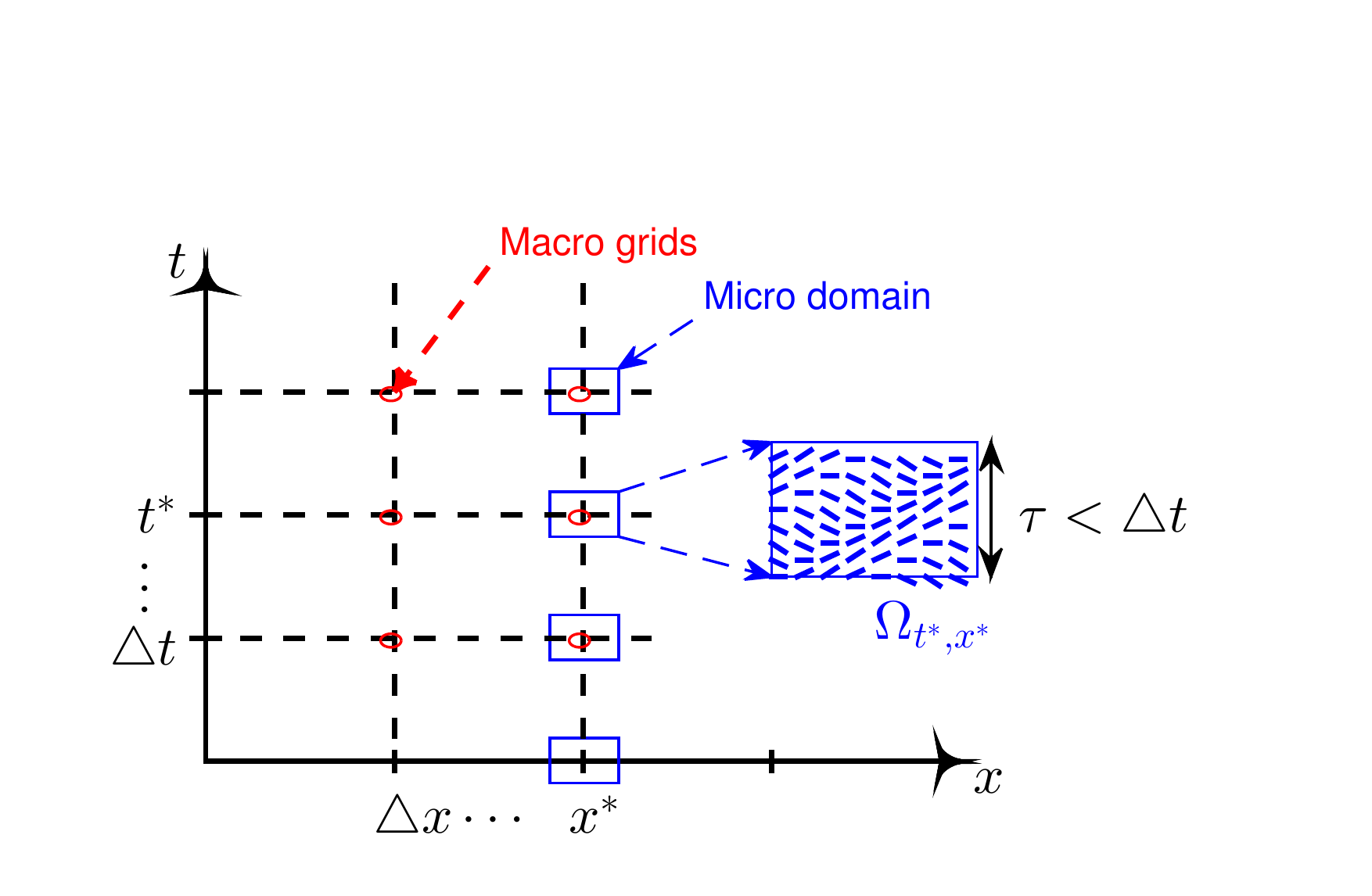}
        \includegraphics[width=0.47\textwidth]{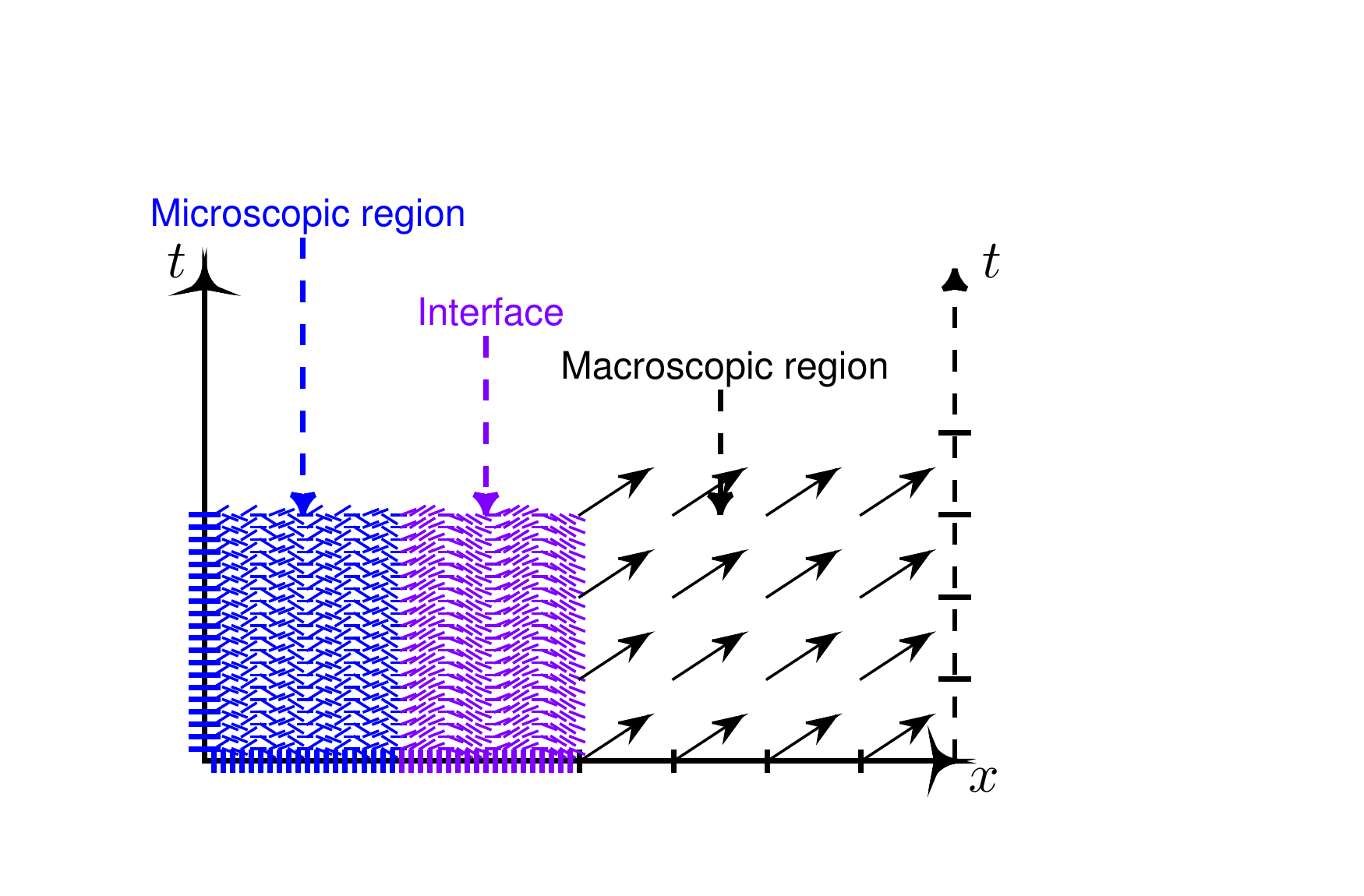}
    \caption{(Left) A multiscale strategy based on upscaling: The macroscopic model lacks some information in the center of the micro boxes, e.g. at $(t^{*},x^{*})$. The missing effective parameter values are then computed by carrying out local, in time and space, microscopic simulations in small boxes, e.g. $\Omega_{t^{*},x^{*}}$. The micro problems should be synchronised with the corresponding coarse scale data and the effective quantities are computed and upscaled from the micro to the macro level. (Right) A multiscale strategy based on domain partitioning: Different mathematical laws are valid at different regions. The macroscopic region does not see the small scale variations in the microscopic domain. These two regions are connected by an interface. The interface acts as a transition region and conditions are often imposed to ensure the consistency between the microscopic and the macroscopic quantities.}
                \label{Fig_MultiscaleApproaches}
\end{figure}

In this article, as a first step towards designing general multiscale
methods such as the upscaling strategy illustrated in Figure
\ref{Fig_MultiscaleApproaches} (the left schematic), we propose and
analyse an upscaling strategy based on heterogeneous multiscale
methods (HMM) \cite{E_Engquist_1}, to couple disparate scales
in the LL equations. The algorithm assumes a macro model where some
data in the model are missing. These data are then computed and
upscaled by carrying out simulations in localised, in time and space,
microscopic domains. The analysis part of this paper is limited to temporal
upscaling only, and therefore the theoretical setting consists of a case where 
the effective field $\bH$ includes only the influence of a time-dependent external field $\tilde{\bH}$ and
no particle interactions are involved. Since we do not have any interaction, such
as exchange, among the particles, we regard all particles as being
identical and rewrite the LL equation \eqref{LLG_eqn} as
\begin{equation} \label{LLG1_eqn}
\begin{array}{lll}
\dfrac{d}{dt} \bm^{\e}(t)  = -\beta \bm^{\e} \times   \bH^{\e} -  \gamma \bm^{\e} \times \left(  \bm^{\e} \times  \bH^{\e} \right). 
\end{array}
\end{equation}

In this setting the dynamics  is driven
only by an external field $\bH^{\e}(t) = \bH(t,t/\e)$, which may possess slow and fast
variations at the same time. The parameter $\e \ll 1$ denotes the
wave-length of the rapid variations. The high frequency external field is a realistic assumption; relevant for applications 
in ferromagnetic resonance, see e.g. \cite{Kittel_1}.  For the analysis, we will assume
that $\bH(t,\cdot)$ is a smooth and periodic function. Note that the above simplification of the effective field $\bH$ is 
only for the analysis, and numerical results are also provided for the case when the exchange interactions are present 
in the model problem.

The paper is structured as follows. In Section
\ref{Sec_Preliminaries}, we introduce the notation and preliminary
results which will be used later in the analysis.  In Section
\ref{Sec_MultiscaleMethods}, we give a general overview of HMM for
multiscale ordinary differential equations (ODEs) and we also present
our multiscale algorithms for a single magnetic particle as well as a chain of interacting magnetic particles. Section \ref{Analysis_Sec} is devoted to the
analysis of the method where a fully discrete estimate for the
difference between the HMM solution and the exact local average is
given. Finally, we conclude our paper by presenting numerical examples corroborating our theoretical statements.


\section{Preliminaries} 
\label{Sec_Preliminaries}
\subsection{Notation}

We will use calligraphic fonts, e.g. ${\bf{\mathcal{H}}}(t)$, to
denote matrices or matrix functions. Vector functions are represented
by a bold face letter, e.g. ${\bf{m}}(t)$. Macroscopic and microscopic
step-sizes will be denoted by $\triangle t$ and $\delta t$,
respectively. The space of functions with bounded total variations in $\mathbb{R}$ is written
as $BV(\mathbb{R})$. The Euclidian $2$-norm is written as $\left| \bm
\right|_2 = \sqrt{m_1^2 + m_2^2 + m_3^2}$ and we denote the
corresponding inner-product by $\langle {\bf{u}},{\bf{v}} \rangle =
\sum_{i=1}^{d} {\bf{u}}_i {\bf{v}}_i$. Unless otherwise specified, $|
\cdot |$ is to mean $|\cdot|_2$. Throughout the paper, the constant
$C = O(1)$ is a generic constant whose value may change in subsequent
occurences.

\subsection{Utility results}

The upscaling strategy that is summarised in the next section is
heavily based on the notion of local averages. Assume that we have an
oscillatory function $f^{\e}(t):=f(t/\e)$ for $t \in [0,T]$, where $f \in L^{\infty}([0,1])$
is a $1$-periodic function and $\e \ll 1$. The function $f(t/\e)$ is
then $\e$-periodic. To compute the local average we introduce an
interval $I_{\tau}:=[-\tau/2,\tau/2]$, where $\tau > \e$. A naive
local averaging approach is to integrate the function $f^{\e}$ over
the interval $I_{\tau}$. The following estimate for the error of this
naive averaging can easily be shown.
\begin{equation} \label{Naive_Averaging_eqn}
\left| \dfrac{1}{\tau} \int_{-\tau/2}^{\tau/2} f(s/\e) \; ds - \int_{0}^{1} f(s) \; ds \right|  \leq C \dfrac{\e}{\tau}.
\end{equation}
From this inequality it is clear that for the error to be small, the condition $\tau > \e$ should hold.
In this paper, $\tau$ stands for the size of the microscopic
simulation box in time and it should practically be chosen of the same order as
$\varepsilon$, but the theory is not limited by this assumption. The error in \eqref{Naive_Averaging_eqn} can be
very large when $\tau = O(\e)$. To reduce this error, we introduce the
space of kernels $ \mathbb{K}^{p,q}$ and we say that $K \in
\mathbb{K}^{p,q}$ if $K$ has a compact support in $[-1/2,1/2]$ and
\begin{itemize}
\item $K(t) = K(-t)$
\item $K^{(q+1)} \in BV(\mathbb{R})$ 
\item $K$ has $p$ vanishing moments:
\begin{equation*}
\int_{\mathbb{R}} K(t) t^r dt = \begin{cases}
1 & r=0, \\
0, & 0< r \leq p.
\end{cases}
\end{equation*}
\end{itemize}

Note that $K^{(q+1)}$ is the $(q+1)th$ weak derivative and that a
constant kernel belongs to $K^{1,-1}$, see also Figure \ref{Fig_KernelExample} for a few other examples. To localise the kernel we
introduce a scaled version $K_{\tau}(t) := 1/\tau K(t/\tau)$, with $K
\in \mathbb{K}^{p,q}$. The local averages can then be computed by
taking a weighted average of the function $f^{\e}$ with $K_{\tau}$
over the domain $I_{\tau}$.
The following lemma from \cite{Arjmand_Runborg_1} (see also \cite{Arjmand_Runborg_2} for numerical experiments) shows how the convergence rate in \eqref{Naive_Averaging_eqn} improves upon using such averaging kernels. 
\begin{Lemma}[Lemma 2.3 in \cite{Arjmand_Runborg_1}] \label{Averaging_Lemma} Let $f \in L^{\infty}([0,1])$ be a $1$-periodic function and $K \in \mathbb{K}^{p,q}$. Then with $\bar{f} = \int_0^1 f(s) ds$ we have
$$
\left| \int_{-\tau/2}^{\tau/2} K_{\tau}(s) f(s/\e) \; ds - \bar{f} \right| \leq C \left|  f \right|_{\infty}  \left( \dfrac{\e}{\tau}  \right)^{q+2},
$$
where the constant $C>0$ is independent of $\e$ and $\tau$ but may depend on $K,f,p$ or $q$.
\end{Lemma}

In Figure \ref{Fig_KernelExample}, three examples of kernels with different $(p,q)$ pairings are given.

\begin{figure}[h]
    \centering
      \includegraphics[width=0.75\textwidth]{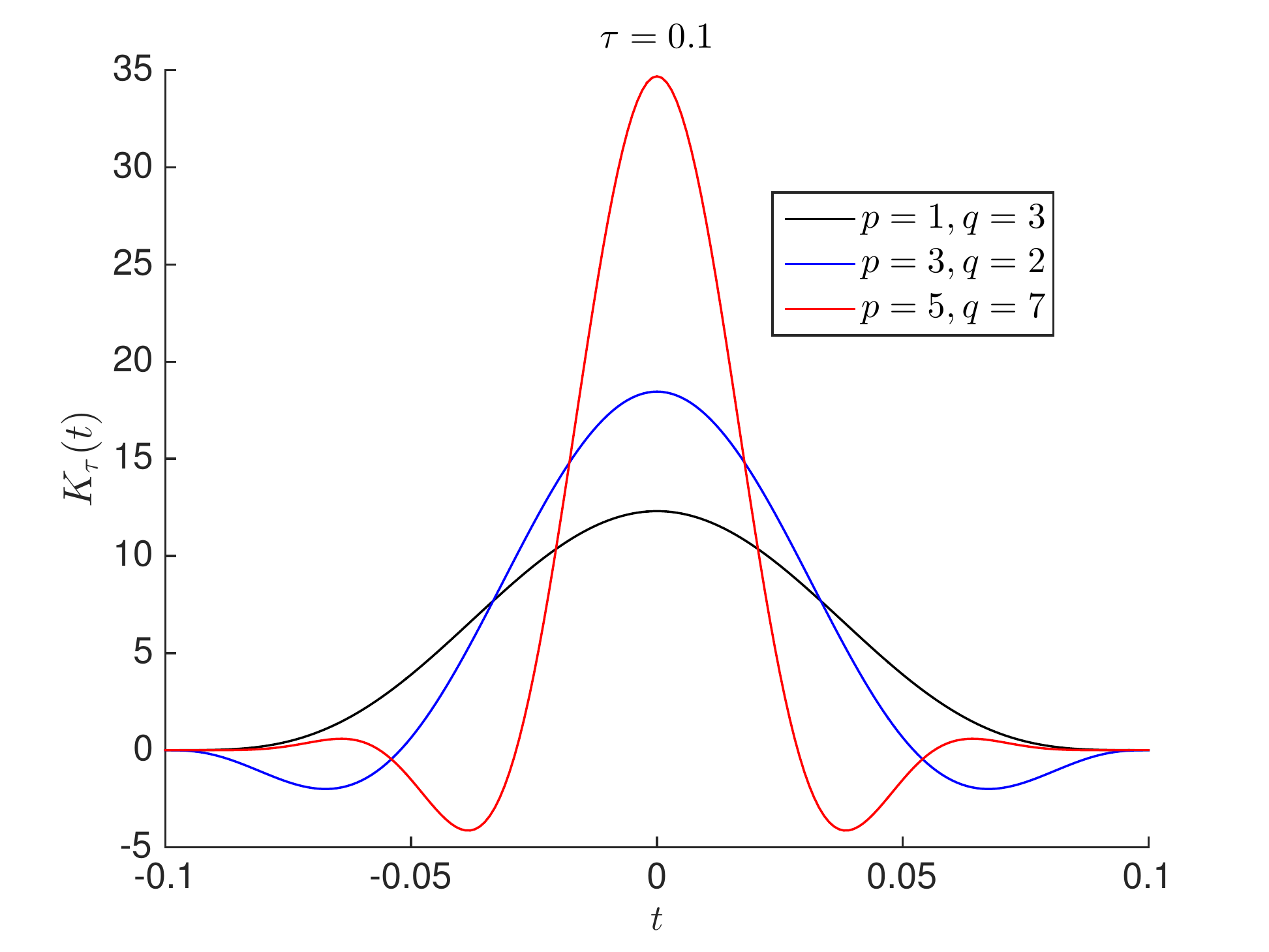}
     \caption{Scaled kernels $K_{\tau}$ with $\tau = 0.1$, and different $p,q$ values are depicted.}
        \label{Fig_KernelExample}
\end{figure}


A well-known property of the LL equation \eqref{LLG_eqn} is the conservation of the length of $\bm$. In other words, if $| \bm(0)|_{2} = 1$, the vector $\bm(t)$ will stay on the unit sphere $\mathbb{S}^{2}$ in $\mathbb{R}^{3}$ for all $t \in [0,T]$ . This can be seen by taking the inner-product of \eqref{LLG_eqn} by $\bm(t)$. From  a numerical point of view, it is then desirable to use discretisation schemes that respect this property. We refer to \cite{Cimrak,Cervera2007} for a survey of common time stepping schemes for the LL equation, to \cite{Prohl_Book} for an analytical treatment of general computational approaches and to \cite{Kruzik_Prohl} for the treatment of relaxation based models used in stationary micromagnetism. One commonly used time-stepping scheme that preserves the magnetisation amplitude along with other invariants such as the Hamiltonian and the Lyapunov structures is the midpoint rule, see \cite{Daquino_etal} for the analysis of the midpoint rule and \cite{McLachlan_Modin_Olivier} for a variant of it. The midpoint rule for \eqref{LLG_eqn} is 
$$
\bm_{n+1} = \bm_{n}  - \beta \triangle t  \bm_{n+1/2} \times \bH_{n+1/2} - \gamma \triangle t \bm_{n+1/2} \times \left( \bm_{n+1/2} \times \bH_{n+1/2}  \right),
$$
where ${\bf{u}}_{n+1/2} = ({\bf{u}}_n + {\bf{u}}_{n+1})/2$ for ${\bf{u}} = \bm$ or ${\bf{u}} = \bH$. In this paper, we will use the mid-point rule as our choice of temporal discretisation for the LL equation. However, we emphasise that the multiscale method that we describe in the next section can easily be modified for any integration method with no conceptual change in the algorithm.




\section{The multiscale method}
\label{Sec_MultiscaleMethods}
\subsection{Heterogeneous multiscale methods}

HMM was originally proposed by E and Engquist \cite{E_Engquist_1} as a
general framework for designing multiscale algorithms for coupling
mathematical models at different scales. For a recent account of
applications of HMM in various fields, see the survey article
\cite{Abdulle_E_etal}. Here we refer to
\cite{Engquist_Tsai_1,Ariel_Engquist_Tsai_1,Ariel_Engquist_Tsai_2,Ariel_etal,E_1}
for the analysis and developments in the context of ordinary
differential equations (ODEs) with multiple scales. For the analysis
of HMM in the stochastic setting, see
e.g. \cite{VandenEijden_1,E_Liu_VandenEijnden_1}. Before proceeding
further to the presentation of our multiscale method, we present the
general idea behind HMM. For this, we introduce
\begin{equation*}
\left( \mathcal{K} \bm^{\e} \right)(t_a) := \int K_{\tau}(t - t_a) \bm^{\e}(t) \; dt.
\end{equation*}
We can then see that $\mathcal{K} \partial_t  = \partial_t \mathcal{K}$. Next, we apply the operator $\mathcal{K}$ to both sides of the equation $\eqref{LLG1_eqn}$ and see that, with $\bM(t) :=  \left( \mathcal{K} \bm^{\e} \right)(t) $,
\begin{equation*}
\dfrac{d}{dt} \bM(t)  = -\beta \left( \mathcal{K}  \left( \bm^{\e} \times \bH^{\e} \right)\right)(t)  - \gamma \left( \mathcal{K} \left(  \bm^{\e} \times \left(  \bm^{\e} \times \bH^{\e}\right) \right) \right) (t).
\end{equation*}
The macro variable $\bM$ is smooth since by definition the fluctuations in $\bm^{\e}$ are filtered out. Moreover, the right hand side is given by a local average of the original LL which oscillates around and damps towards the external field $\bH^{\e}$. This suggests that we may assume a macro model of the form:
\begin{equation}\label{MacroModel_HMM_eqn}
\dfrac{d}{dt} \bM(t)  = \bF(t,\bM),
\end{equation}
where the term $\bF$ includes a local average of the time derivative $\frac{d}{dt} \bm^{\e}$, appearing in \eqref{LLG1_eqn}. In the next section, we propose a somewhat similar strategy where the damping term is modelled explicitly in the macro model. We will show later, by analytical and numerical means, that the proposed strategy captures the right macroscopic dynamics, and we quantify the error that arises from a multiscale coupling.


\subsection{HMM for a single particle}
\label{Sec:HMM_SingleSpin}
 The effective equation on the macro level assumes the following form:
\begin{equation}\label{Macro_Problem_eqn}
\text{Macro Model: } 
\left\{
\begin{array}{ll}
\dfrac{d}{dt} \bM(t) = \bF(t,\bM) +  \dfrac{\gamma}{\beta} \bM \times \bF(t,\bM), \quad t \in (0,T], \\
 \bM(0) = \bm^{\e}(0),
 \end{array}
 \right.
\end{equation}
where $\bF$ is the missing data in the model. Note that, unlike the
model \eqref{MacroModel_HMM_eqn}, the macro model
\eqref{Macro_Problem_eqn} includes also a damping term . As the
macro-solver, we use the midpoint rule due to the mentioned
conservation properties it possesses. In principle, the HMM introduced in this section can be combined with any time-stepping scheme with no change in the macroscopic and microscopic models. In this paper, we will carry out the analysis for the implicit midpoint rule. The macro-solver uses a macro
step-size $\triangle t$ and reads as
\begin{equation} \label{Macro_Solver_eqn}
\text{Macro Solver: }
\left\{
\begin{array}{ll}
\bM_{n+1} = \bM_n  + \triangle t \bF_{n+1/2} + \dfrac{\gamma}{\beta} \triangle t \bM_{n+1/2} \times \bF_{n+1/2}, \\
\bM_0 = \bm^{\e}(0),
\end{array}
\right.
\end{equation}
where $\bF_{n+1/2}:=\bF\left(  t_{n+1/2} ,  \bM_{n+1/2} \right)$,  $t_{n+1/2}:= (t_n + t_{n+1})/2$ and $\bM_{n+1/2} = (\bM_n  + \bM_{n+1})/2$. To close the system one needs to compute the right hand side $\bF(t_a,\bM)$. To achieve this we first solve the micro model:
\begin{equation} \label{Micro_Problem_eqn}
\text{Micro Model: }
\begin{array}{lll}
\dfrac{d}{dt} \bm^{\e}(t)  = -\beta \bm^{\e}(t) \times   \bH^{\e}(t),  \quad t \in I_{\tau}^{\pm},\\
\bm^{\e}(t_a)  = \bM,
\end{array}
\end{equation}
where $I_{\tau}^{+}:= (t_a,t_a + \tau/2]$ and $I_{\tau}^{-}:=[-\tau/2 + t_a, t_a)$. Then, $\bF(t_a,\bM)$ can be computed by:
\begin{equation} \label{Upscaling_eqn}
\text{Upscaling: } \bF(t_a,\bM) = \left( \mathcal{K}  \dfrac{d}{dt} \bm^{\e} \right)(t_a):=\int_{t_a-\tau/2}^{t_a+\tau/2} K_{\tau}\left( s-t_a \right) \dfrac{d}{ds} \bm^{\e}(s) \; ds.
\end{equation}

In practice we choose $\tau = O(\e)$ (but $\tau > \e$) and therefore the computational cost of solving the micro model \eqref{Micro_Problem_eqn} is independent of the small scale parameter $\e$. Nevertheless, the theory and the method itself are not restricted by this assumption. Note that the initial data of the micro problem is given by the coarse scale solutions. This is necessary for synchronising the micro-problem with the current macroscopic quantities. Moreover, the macro stepsize is chosen larger than the size of the microscopic simulation box, i.e., $\tau < \triangle t$, so that the overall cost of the multiscale method is cheaper than the cost of a direct numerical approximation of the full multiscale problem. 
\begin{Remark}
We emphasize that for computation of the flux \eqref{Upscaling_eqn}, the micro problem must be solved forward and backward locally in time. 
\end{Remark}
The micro model does not include the damping term. This does not cause any big change in the dynamics of the macro-scale variable $\bM$ owing  to the fact that the solution $\bm^{\e}$ of the micro model \eqref{Micro_Problem_eqn} with or without damping would precess around the field $\bH^{\e}$ in both cases and the high frequency variations are filtered out in the upscaling step \eqref{Upscaling_eqn}. It is possible to include the damping term in the micro problem but this requires removing the damping in the macroscopic model \eqref{Macro_Problem_eqn}. It is also possible to carry out the local averaging in the upscaling \eqref{Upscaling_eqn} locally forward in time instead, i.e., in the interval $[0,\tau]$. Such a procedure will result in a qualitatively similar macroscopic  solution but may lead to a significant deterioration in the order of convergence. The main advantage with using the conservative micro model \eqref{Micro_Problem_eqn} is that the nonlinearity is treated only in the macroscopic equation \eqref{Macro_Problem_eqn} which leads to significant simplifications in relation with the convergence analysis for the upscaling step and the numerical computation of the micro solution. Note that the macro-solver is also nonlinear, in $\bM_{n+1}$, due to our choice of discretisation.
\begin{Remark}
The micro problem \eqref{Micro_Problem_eqn} can be posed instead over the domain $(0,\tau/2]$ by shifting the argument of the field $\bH^{\e}$ by $t_a$:
 \begin{equation} \label{Micro_Problem1_eqn}
\begin{array}{lll}
\dfrac{d}{dt} \bm^{\e}(t)  = -\beta \bm^{\e}(t) \times   \bH^{\e}(t + t_a),  \quad t \in \tilde{I}_{\tau}^{\pm}\\
\bm^{\e}(0)  = \bM, 
\end{array}
\end{equation}
where $\tilde{I}_{\tau}^{+}:=(0,\tau/2]$ and $\tilde{I}_{\tau}^{-}:=(-\tau/2,0]$. Accordingly, the flux \eqref{Upscaling_eqn} can be rewritten as
\begin{equation} \label{Upscaling1_eqn}
\bF(t_a,\bM) = \left( \mathcal{K}  \dfrac{d}{dt} \bm^{\e} \right)(0):=\int_{-\tau/2}^{\tau/2} K_{\tau}\left( s \right) \dfrac{d}{ds} \bm^{\e}(s) \; ds.
\end{equation}

\end{Remark}

\subsection{Extension of HMM to a chain of interacting magnetic particles }
\label{Sec:HMMChain}
In this section, the HMM algorithm described in Section \ref{Sec:HMM_SingleSpin} is adapted for a system of interacting particles. It is assumed that 
the magnetic particles are sitting on a number of equidistant discrete points $\{ x_i  = i \delta x \}_{i=0}^{(r+\ell)L}$, where $\delta x$ represents the distance between two adjacent particles, $r \in \mathbb{Z}^{+}$ and $\ell \in \mathbb{N}$ are two integers (whose roles will be made clear), and $N = (r+\ell)L$ is the total number of particles. To have a well-defined problem setting, periodic boundary condition is assumed for the chain. The full microscopic equation reads as 
\begin{equation} \label{FullMicro_eqn}
\dfrac{d \bm^{\e}_i}{dt}  = - \beta \bm^{\e}_i \times  \bH^{\e}_i - \gamma \bm_i^{\e} \times \left(  \bm_i^{\e} \times \bH^{\e}_i  \right), \quad i=1,\ldots,N,
\end{equation}
where $ \bH^{\e}_i =  \left(   \sum_{j}  J_{ij} \bm_{j}^{\e}    + \tilde{\bH}^{\e}_i(t) \right)$ and $\tilde{\bH}^{\e}_i(t)$ is a high-frequency external field. The index $i$ in $\tilde{\bH}_i^{\e}$ is to allow for spatially nonuniform external fields. The macroscopic model (to be introduced soon) is built on a coarse mesh, whereby the magnetisation on every coarse grid is defined as the local average of $(2 r + 1)$ microscopic magnetic moments.  For this, we first define a coarse grid $\{ X_{I}  = I (r+ \ell)  \delta x\}_{I=0}^{L}$ with $L \ll N$, which implies a far less degrees of freedom in comparison to the full microscopic equation \eqref{FullMicro_eqn}. The macroscopic magnetisation $\bM_I$ at the point $X_I$ is defined as 

\begin{align} \label{DefnMacroM_eqn}
\bM_I(t) = \sum_{j=-r}^{r} K_{\eta}(x_{I(r+ \ell)+j}   -   x_{I(r+\ell)})  \left(  K_{\tau} \ast \bm^{\e}_{I(r+\ell)+j} \right)(t)=:  \left( \mathcal{K}_{\tau,\eta} \ast \bm^{\e} \right)(X_I,t), 
\end{align}
where $\eta = 2 r \delta x$ is the size of the local spatial averaging domain and $\tau$ is that of a temporal averaging. From formula \eqref{DefnMacroM_eqn}, it is evident that between two consecutive macroscopic points, a total number of $\ell$ magnetic moments are skipped while averaging. The macromodel is similar to \eqref{Macro_Problem_eqn} and takes the form

\begin{equation}\label{MacroProblemChain_eqn}
\text{Macro Model: } 
\left\{
\begin{array}{ll}
\dfrac{d}{dt} \bM_I(t) = \bF_I(t,\bM_{\tilde{I}}) +  \dfrac{\gamma}{\beta} \bM \times \bF_I(t,\bM_{\tilde{I}}), \quad t \in (0,T], \\
 \bM_I(0) = \left( \mathcal{K}_{\tau,\eta} \ast \bm^{\e}\right)(X_I, 0),
 \end{array}
 \right.
\end{equation}
where $\tilde{I} = \{I-1,I,I+1 \}$. To close the macro problem,  $\bF_I(t_a,\bM_{\tilde{I}})$ needs to be computed. To do this, we first solve the micro problem

\begin{equation} \label{MicroProblemChain_eqn}
\text{Micro Model:}
\left\{
\begin{array}{lll}
\dfrac{d}{dt} \bm^{\e}_{Ir^{\prime}  + j} (t)  = -\beta \bm^{\e}_{I r^{\prime}  + j} \times \bH^{\e}_{Ir^{\prime} + j}(t), \quad t \in I^{\pm}_{\tau}, \quad j =-r,\ldots,r \\
\bm^{\e}_{Ir^{\prime} + j}(t_a)  = \hat{\bM}(x_{Ir^{\prime} + j}) \\
\bm^{\e}_{Ir^{\prime} - r}   = \hat{\bM}(x_{Ir^{\prime}  - r}), \quad \bm^{\e}_{Ir^{\prime} + r}   = \hat{\bM}(x_{Ir^{\prime}  + r}),
\end{array}
\right.
\end{equation}
where $I^{\pm}_{\tau}$ is given by \eqref{Micro_Problem_eqn}, $r^{\prime} = r + \ell$, and $\hat{\bM} = \pi_2 \bM/|\pi_2 \bM |$ is obtained by a normalised second order polynomial interpolation of the macroscopic solutions $\bM_{\tilde{I}}$ for $\tilde{I} = I-1,I,I+1$. The last step is to upscale the quantity $\bF_I(t_a,\bM_{\tilde{I}})$ by
\begin{equation}\label{UpscalingChain_eqn}
\text{Upscaling:} \quad   \bF_{I}(t_a,\bM_{\tilde{I}})  =  \left(  \mathcal{K}_{\tau,\eta} \ast  \dfrac{d}{dt} \bm^{\e}  \right)\left( X_I,t_a \right). 
\end{equation}
\begin{Remark} Note that the micro problems \eqref{MicroProblemChain_eqn} are independent from each other, and hence they can be solved in parallel. 
\end{Remark}
\begin{Remark} In the formulation of the micro problem \eqref{MicroProblemChain_eqn}, the boundary conditions are obtained from the macroscopic solutions. When the micro solution has high frequency variations in space, the high frequencies which are reflected from the boundary of the micro problem may pollute the interior solution. In this scenario, the micro problem may be modified to also include a damping layer near the boundary, whereby high frequency components in the solution are damped out. Such a strategy is developed in \cite{Poluektov_Eriksson_Kreiss} in the setting of domain partitioning type multiscale methods, but the HMM algorithm, proposed here, can be easily modified to include this damping layer too, if necessary. 
\end{Remark}


\section{Analysis}
\label{Analysis_Sec}
For the analysis, we consider a single magnetic particle, and without loss of generality we assume that $\beta=1$. The LL equation with $\beta=1$ can be written as
\begin{equation} \label{LLG_Beta1_eqn}
\dfrac{d}{dt} \bm^{\e}(t) = -\mathcal{H}^{\e}(t) \bm^{\e}(t) - \gamma \bm^{\e}(t) \times \left( \bCalH^{\e}(t) \bm^{\e}(t) \right),
\end{equation}
where $\bCalH^{\e}(t) \in \mathbb{R}^{3 \times 3}$ for all $t \in [0,T]$ is a skew-symmetric matrix function given by
\begin{equation*}
\mathcal{H}^{\e}(t)=  \left[ \begin{array}{ccc}
0 & H_3^{\e}(t) & -H_2^{\e}(t) \\
-H_3^{\e}(t) & 0 & H_1^{\e}(t) \\
H_2^{\e}(t) & -H_1^{\e}(t) & 0 \end{array} \right]. 
\end{equation*}

For the analysis, we assume that $\bCalH^{\e}(t) = \bCalH(t,t/\e)$, where the field $\bCalH(t,\cdot)$ is $1$-periodic and $\bCalH_{ij} \in C^{\infty}([0,T]\times[0,1])$, so that we can safely assume that for all ${\bf{v}} \in \mathbb{R}^3$, and $k = 0,1$, the following bounds hold
\begin{equation}\label{Assumptions_eqn}
\left| \partial_t^{k} \bCalH(t,s)  {\bf{v}}  \right|_2 \leq c_1 \left|   {\bf{v}}  \right|_2,  \quad \forall t \in [0,T] \text{ and } s \in \mathbb{R}.
\end{equation}
By periodicity we can write
$$
\mathcal{H}^{\e}(t + t_a) = \mathcal{H}(t+t_a , t/\e + r ), \text{ where } r = \{  t_a/\e  \},
$$
and $\{ z \}$ represents the fractional part of $z$. To simplify the notation, we will write $\mathcal{H}_{t_a,r}(t , t/\e) :=\mathcal{H}(t+t_a , t/\e + r )$. 

In practice, in the ODE setting, it is often the case that
$\bCalH(t,t/\e) = \bCalH_{0} + \bCalH_1(t/\e)$ where $\bCalH_{0}$ is
constant and hence $\partial_t \bCalH(t,z) = 0$. However, we will keep
the assumptions as above so that they are suitable also for more
general cases when the solution $\bm^{\e}$ might have a non-trivial
dependence on the slow and fast temporal scales.

From the classical theory of averaging for oscillatory ODEs, see
e.g. \cite{Bogoliubov,Pavliotis_Stuart}, it follows that, when
$\bCalH(t,\cdot)$ is $1$-periodic, the exact average magnetisation,
$\bar{\bm}$, is driven by the effective field
$\bar{\bCalH}(t):=\int_{0}^{1} \bCalH(t,s) \; ds$ and the effective
equation reads as

\begin{equation} \label{Effective_eqn}
\dfrac{d}{dt} \bar{\bm}(t)  = -\bar{\bCalH}(t) \bar{\bm}(t) - \gamma \bar{\bm}(t) \times \left( \bar{\bCalH}(t) \bar{\bm}(t) \right) , \quad  \bar{\bm}(0) = \bm^{\e}(0).
\end{equation}

Our aim in the next section is to show that the HMM solution $\bM$ in \eqref{Macro_Problem_eqn} approximates the solution $\bar{\bm}$ of the effective equation \eqref{Effective_eqn} without assuming any knowledge about the effective field $\bar{\bCalH}$. The equation \eqref{Effective_eqn} is needed only for the analysis and the HMM does not use it. 

\begin{Remark} In what follows, we assume that the initial magnetisation $|\bM_0| = 1$. Figure \ref{Fig_MagnetisationLength} shows that the HMM introduces an error at later times and therefore the magnetisation length is not preserved exactly. See Remark \ref{Magnetisation_Amplitude}, where we give an upper bound for this error.
\end{Remark}

\subsection{Upscaling error}

Our main result in this section is the following theorem which shows
that the upscaling step in the HMM strategy captures the right
macroscopic quantity.

\begin{Theorem} \label{Upscaling_Estimate_Thm}Let $\bar{\mathcal{H}}(t):= \int_{0}^{1} \mathcal{H}(t,s) \; ds$. Suppose that $\bm^{\e}(t)$ solves the micro problem \eqref{Micro_Problem1_eqn}. Moreover, suppose that the assumptions \eqref{Assumptions_eqn} hold. Then $\bF(t_a,\bM)$ given by the upscaling step \eqref{Upscaling1_eqn} satisfies
\begin{equation*}
\left|  \bF(t_a,\bM) +  \bar{\mathcal{H}}(t_a)  \bM \right| \leq  C_1 \left(  \tau + \left( \dfrac{\e}{\tau} \right)^{q+2} \right) \left| \bM \right|,
\end{equation*}
where $C_1$ is independent of $\e$, and $\tau$ but may depend on $p,q$ or $K$.
\end{Theorem}

\begin{proof}
We start by rescaling $\bm^{\e}$ to bring the $O(\e)$ oscillations back to $O(1)$ time scales. For this we introduce the fast time scale $\theta = t/\e$
$$
\bm(\theta; \e ) := \bm^{\e}(t). 
$$
Then we see that from \eqref{LLG_Beta1_eqn}
\begin{equation} \label{Scaled_m_eqn}
\partial_t \bm(t;\e) = -\e \mathcal{H}_{t_a,r}(\e t, t) \bm(t;\e).
\end{equation}
We now Taylor expand \footnotemark \footnotetext{Here, by the assumption \eqref{Assumptions_eqn}, $\mathcal{H}$ is smooth. It is assumed that $\bm$ is also smooth enough to allow for a Taylor's expansion.} $\bm(t;\e)$ in terms of $\e$
$$
\bm(t;\e) = \bm_0(t) + \e \bm_1(t)  + \dfrac{\e^2}{2} \bm_2(t) +  \cdots, 
$$
where  $\bm_j(t) := \partial_{\e}^{j} \bm(t;\e) |_{\e=0}$. From \eqref{Scaled_m_eqn} and the definition of $\bm_j$ with $j=0$, we see that
\begin{equation}\label{Term_m0_eqn}
\dfrac{d}{dt} \bm_0(t) = \partial_t \bm(t;0)  = 0, \text{ and  }  \bm_0(0)  = \bm^{\e}(0) = \bM. 
\end{equation}
Hence $\bm_0(t) =\bM$. Moreover, for $j=1$, we differentiate \eqref{Scaled_m_eqn} with respect to $\e$ and find that
\begin{eqnarray*}
\partial_t \partial_{\e} \bm(t;\e) &=& - \partial_{\e} \left(  \e \mathcal{H}_{t_a,r}\left( \e t, t \right) \bm\left( t;\e \right)  \right) \\
&=& - \mathcal{H}_{t_a,r}(\e t,t) \bm(t;\e)  - \e t  \mathcal{H}_{t_a,r}^{(1)}(\e t, t) \bm(t;\e)  - \e \mathcal{H}_{t_a,r}(\e t, t) \partial_{\e} \bm(t; \e),
\end{eqnarray*}
where $ \mathcal{H}_{t_a,r}^{(k)}(t, z): = \partial_t^{k} \mathcal{H}_{t_a,r}(t, z)$. Upon putting $\e=0$, we obtain
\begin{equation}\label{m1_eqn}
\dfrac{d}{dt} \bm_1(t)  = -\mathcal{H}_{t_a,r}(0,t) \bm_0(t), \quad \bm_1(0)=0.
\end{equation}
We now define $\be_{tail}(t;\e): =  \bm(t;\e) - \left( \bm_0(t) + \e \bm_1(t) \right)$. We can then write
\begin{eqnarray*}
\dfrac{d}{dt} \bm^{\e}(t) &=& \partial_{t} \left( \bm_0(t/\e) + \e \bm_1(t/\e) \right) + \partial_t \be_{tail}(t/\e;\e) \\
&=& \partial_{\theta} \bm_1(t/\e)  + \partial_t \be_{tail}(t/\e;\e) = -\mathcal{H}_{t_a,r}(0,t/\e) \bm_0(t/\e) + \partial_t \be_{tail}(t/\e;\e) \\ &=& -\mathcal{H}_{t_a,r}(0,t/\e) \bM + \partial_t \be_{tail}(t/\e;\e).
\end{eqnarray*}
From the last equality, it follows that
\begin{align*}
\left|  \bF(t_a,\bM) +  \bar{\mathcal{H}}(t_a)  \bM \right|  &:= \left| \int_{-\tau/2}^{\tau/2} K_{\tau}\left( s \right) \partial_s \bm^{\e}(s) \; ds  + \bar{\mathcal{H}}(t_a)  \bM  \right| \\ &\hspace{-3cm}= \left| -\int_{-\tau/2}^{\tau/2} K_{\tau}\left( s \right) \mathcal{H}_{t_a,r}(0,s/\e) \bM \; ds  + \int_{-\tau/2}^{\tau/2} K_{\tau}\left( s \right) \partial_s \be_{tail}(s/\e;\e) \; ds  + \bar{\mathcal{H}}(t_a)  \bM  \right| \\ &\hspace{-3cm}\leq
\underbrace{\left| - \int_{-\tau/2}^{\tau/2} K_{\tau}\left( s \right) \mathcal{H}_{t_a,r}(0,s/\e)  \; ds \bM +   \bar{\mathcal{H}}(t_a)  \bM \right|}_{E_{averaging}}  + \underbrace{\left|  \int_{-\tau/2}^{\tau/2} K_{\tau}\left( s \right) \partial_s \be_{tail}(s/\e;\e) \; ds \right| }_{E_{tail}} 
\end{align*}
By Lemma \ref{Averaging_Lemma}, we have $E_{averaging} \leq C \left(
\frac{\e}{\tau}\right)^{q+2} \left| \bM \right|$. Moreover, by Lemma
\ref{Tail_Estimate_Lemma} presented below, it follows that $E_{tail}
\leq C \tau \left| \bM \right|$. The proof is completed.
\end{proof}


\subsection{The tail error} 
The following lemma quantifies the tail error that arises from
replacing $\bm$ by $\tilde{\bm}:= \bm_0 + \e \bm_1$, where $\bm_0$ and
$\bm_1$ are given in \eqref{Term_m0_eqn} and \eqref{m1_eqn}
respectively, in the upscaling step.
\begin{Lemma} \label{Tail_Estimate_Lemma} Suppose that the assumptions \eqref{Assumptions_eqn} hold and that 
$$\tilde{\bm}(t;\e):= \bm_0(t) + \e \bm_1(t), \quad \text{ and } \quad
  \be(t;\e):= \bm(t;\e) - \tilde{\bm}(t;\e),$$ where $\bm_0$ and
  $\bm_1$ are given in \eqref{Term_m0_eqn} and \eqref{m1_eqn},
  respectively. Moreover, let $K \in \mathbb{K}^{p,q}$ and $\tau \geq
  \e$. Then
\begin{align} \label{e_Estimate_eqn}
\left|  \int_{-\tau/2}^{\tau/2} K_{\tau}(t) \partial_t \be(t/\e;\e) \; dt  \right|  \leq  C \tau \left| \bM \right|_2,
\end{align}
where $C$ is independent of $\e,\tau$ and $T$ but may depend on $\bCalH_{t_a,r}$.
\end{Lemma}
\begin{proof}
By definition we can write, with $ \delta \mathcal{H}_{t_a,r}:= \mathcal{H}_{t_a,r}(\e t, t) - \mathcal{H}_{t_a,r}(0, t) $,
\begin{align*}
\partial_t \be(t;\e) &= \partial_t \bm(t;\e)  - \partial_t \tilde{\bm}(t;\e)  = \partial_t \bm(t;\e)  - \e \partial_t \bm_1(t) \\ 
&=  -\e \mathcal{H}_{t_a,r}(\e t, t) \bm(t;\e)  + \e \mathcal{H}_{t_a,r}(0,t) \bm_0(t) \\
&=  - \e \mathcal{H}_{t_a,r}(\e t, t) \be(t;\e)  - \e  \mathcal{H}_{t_a,r}(\e t, t) \tilde{\bm}(t;\e) + \e \mathcal{H}_{t_a,r}(0,t) \bm_0(t) \\
&=  - \e  \mathcal{H}_{t_a,r}(\e t, t) \be(t;\e) - \e  \delta \mathcal{H}_{t_a,r}   \bm_0(t)  - \e^2  \mathcal{H}_{t_a,r}(\e t, t) \bm_1(t).
\end{align*}
Moreover, taking the inner-product of the above equation with $\be$ we get
\begin{align*}
\langle  \partial_t \be  , \be \rangle &= -\e \langle  \mathcal{H}_{t_a,r}(\e t, t) \be  , \be \rangle  - \e  \langle \delta \mathcal{H}_{t_a,r}   \bm_0(t) , \be \rangle - \e^2  \langle \mathcal{H}_{t_a,r}(\e t, t) \bm_1(t),  \be \rangle \\
& = - \e  \langle \delta \mathcal{H}_{t_a,r}   \bm_0(t) , \be \rangle - \e^2  \langle \mathcal{H}_{t_a,r}(\e t, t) \bm_1(t),  \be \rangle.
\end{align*}
This implies the inequality
\begin{align} \label{e_ineq}
\dfrac{1}{2} \partial_t \left|  \be(t;\e)   \right|^{2}  & \leq \left( \e \left|  \delta  \mathcal{H}_{t_a,r}   \bm_0(t)   \right| + \e^2 \left|  \mathcal{H}_{t_a,r}(\e t, t) \bm_1(t)   \right| \right) \left|  \be(t;\e)  \right|.
\end{align}
By the assumptions \eqref{Assumptions_eqn} and a Taylor's expansion,
we have $\left| \delta\bCalH_{t_a,r} \bm_0 \right| \approx \e t
\left| \bM \right|$. For the second term, we again use
\eqref{Assumptions_eqn} and \eqref{m1_eqn}. This gives
\begin{align*}
\left|  \mathcal{H}_{t_a,r}(\e t, t) \bm_1(t)   \right|  &\leq  C \left|  \bm_1 (t) \right|  = C  \left| \int_0^{t}  \bCalH(0,s)  \bm_0  \; ds  \right| \\ &\leq C   \int_0^{t}  \left| \bCalH(0,s)  \bm_0 \right| \; ds   \leq C t \left|   \bM \right|.
\end{align*}
From \eqref{e_ineq}, it follows that 
\begin{align} \label{e_Estimate_eqn_2}
\left|  \be(t;\e)  \right|_2 &\leq \int_0^{t}  \e \left|  \delta  \mathcal{H}_{t_a,r}   \bm_0(s)   \right|  + \e^2 \left|  \mathcal{H}_{t_a,r}(\e s, s) \bm_1(s)   \right| \; ds \leq C \e^2 \left(1 +  t^2  \right)\left|  \bM  \right|.
\end{align}
We can now bound the tail error in the upscaling procedure. With $\alpha:= \e/\tau$, we have
\begin{align*}
\left( \mathcal{K} \partial_t \be(t/\e;\e) \right)(0)  &:= \int_{-\tau/2}^{\tau/2} K_{\tau}(t) \partial_t \be(t/\e;\e) \; dt = \dfrac{1}{\tau}  \int_{-\tau/2}^{\tau/2} K(t/\tau) \partial_t \be(t/\e;\e) \; dt \\
&= -\dfrac{1}{\tau} \int_{-\tau/2}^{\tau/2} \dfrac{1}{\tau} K^{\prime}(t/\tau)  \be(t/\e;\e) \; dt = -\alpha \int_{-1/{2\alpha}}^{1/{2\alpha}} \dfrac{1}{\tau} K^{\prime}(t\alpha)  \be(t;\e) \; dt.
\end{align*}
Finally, we use \eqref{e_Estimate_eqn_2} to obtain
\begin{align*}
\left|  \left( \mathcal{K} \partial_t \be(t/\e;\e) \right)(0) \right| \leq \left| K^{\prime} \right|_{\infty} \dfrac{1}{\tau} \max_{t \in [-\frac{1}{2\alpha}, \frac{1}{2\alpha}]}\left| \be(t;\e) \right| \leq C \tau \left| \bM \right|.
\end{align*}
\end{proof}


\subsection{Estimates for the full solution}

In this section, we give a fully discrete error estimate for the
difference between the HMM solution given in \eqref{Macro_Solver_eqn}
and the solution of the exact locally-averaged equation
\eqref{Effective_eqn}. For this, let us denote the HMM solution by $\{
\bM_{n}^{\triangle t, \delta t} \}_{n=0}^{N}$. This is the solution
that is computed by solving the micro and the macro problems in a
discrete setting. We define also a semi-discrete solution $\{
\bM_{n}^{\triangle t} \}_{n=0}^{N}$ which is essentially the HMM
solution but when the micro problem is solved exactly.
Moreover, we introduce $\{ \bar{\bm}^{\triangle t}_j \}_{j=1}^{N}$
which approximates the solution of the effective equation
\eqref{Effective_eqn} by the implicit midpoint rule:

\begin{equation} \label{Discrete_Averaged_Eqn}
\bar{\bm}_{n+1}^{\triangle t}  = \bar{\bm}_{n+1}^{\triangle t}  - \triangle t \bar{\bCalH}(t_{n+1/2}) \bar{\bm}_{n+1/2}^{\triangle t} + \bN_{\gamma}[\bar{\bm}^{\triangle t}_{n+1/2}], \quad  \bar{\bm}^{\triangle t}_0 = \bm^{\e}(0),
\end{equation}
where $\bar{\bm}_{n+1/2} := \dfrac{1}{2} \left(   \bar{\bm}_{n+1} + \bar{\bm}_{n}  \right)$ and 
$$
\bN_{\gamma}[\bar{\bm}^{\triangle t}_{n+1/2}]  := - \gamma \bar{\bm}^{\triangle t}_{n+1/2} \times \left( \bar{\bCalH}(t_{n+1/2}) \bar{\bm}_{n+1/2}^{\triangle t} \right).
$$
On the other hand, we see from equation  \eqref{Macro_Solver_eqn} that the semi-discrete HMM solution satisfies
\begin{align} \label{Macro_Solver_NewNotation}
\bM_{n+1}^{\triangle t}  &= \bM_{n}^{\triangle t}  + \triangle t \bF(t_{n+1/2},\bM_{n+1/2}^{\triangle t}) + \gamma \triangle t \bM_{n+1/2}^{\triangle t} \times  \bF(t_{n+1/2},\bM_{n+1/2}^{\triangle t}) \\
&= \bM_{n}^{\triangle t} + \triangle t \left( - \bar{\bCalH}(t_{n+1/2}) \bM_{n+1/2}^{\triangle t} + E_{n+1/2}^{ups} \right)  + \triangle t \bN_{\gamma}[\bM_{n+1/2}^{\triangle t}] \nonumber \\ &- \gamma \triangle t \bM_{n+1/2}^{\triangle t} \times E^{ups}_{n+1/2}, \nonumber
\end{align}
where $E_{n+1/2}^{ups} :=  \bF(t_{n+1/2},\bM_{n+1/2}) + \bar{\bCalH}(t_{n+1/2}) \bM_{n+1/2}$ satisfies the estimate in Theorem \ref{Upscaling_Estimate_Thm}.
We can then split the overall error into three parts as follows
$$
\left|   \bM^{\triangle t , \delta t}_n  - \bar{\bm}(t_n) \right| \leq \underbrace{\left|   \bM^{\triangle t  , \delta t}_{n}  - \bM^{\triangle t}_n   \right|}_{E_{micro}}  + \underbrace{\left| \bM^{\triangle t}_{n}  - \bar{\bm}_n^{\triangle t} \right|}_{E_{HMM}}  + \underbrace{\left|   \bar{\bm}_n^{\triangle t}  - \bar{\bm}(t_n)  \right|}_{E_{Macro}}.
$$
The first and the last errors are the micro and the macro discretisation errors, which can be bounded in a rather standard way. The HMM error $E_{HMM}$ is, however, related to the upscaling error $E_{n}^{ups}$ which was estimated in Theorem \ref{Upscaling_Estimate_Thm}.  In what follows we will firstly prove a theorem which provides an upper-bound for the error $E_{HMM}$ for the case $\gamma  = 0$.
\begin{Theorem}\label{Theorem1_HMM_Error} Suppose that $\{ \bM_{n}^{\triangle t} \}_{n=0}^{N}$ solves \eqref{Macro_Solver_eqn}  with $\beta = 1$ and $\gamma  = 0$, where the associated micro problem \eqref{Micro_Problem_eqn} is solved exactly. Moreover, assume $K \in \mathbb{K}^{p,q}$ and that $\bar{\bm}_{n+1}^{\triangle t}$ solves \eqref{Discrete_Averaged_Eqn} and $\triangle t \leq 2/3$. Then the error $\be_{n}^{\triangle t} :=\bM^{\triangle t}_{n}  - \bar{\bm}_n^{\triangle t}$ satisfies
\begin{equation*}
| \be_N^{\triangle t}  |  \leq  C \sqrt{T} \left(  \tau  + \left( \dfrac{\e}{\tau} \right)^{q+2}  \right) e^{\frac{3}{4} T} \max_{0 \leq n \leq N} | \bM_{n+1/2}^{\triangle t} | ,
\end{equation*}
where $T = N \triangle t$, and $C$ is a constant independent of $\e, \tau$ and $T$ but may depend on $\bCalH,p ,q , K$.
\end{Theorem}
\begin{proof} For simplicity we drop $\triangle t$  in the notation and instead write $\be_{n}$. We can see, from \eqref{Discrete_Averaged_Eqn} and \eqref{Macro_Solver_NewNotation}, that
$$
\be_{n+1} = \be_{n} - \triangle t \bar{\bCalH}(t_{n+1/2}) \be_{n+1/2} + \triangle t E_{n+1/2}^{ups}, \quad \be_0 = 0.
$$
Next we rearrange the above scheme and take the inner-product with $\be_{n+1/2}$. This gives
\begin{align*}
\langle  \be_{n+1} - \be_{n}  , \be_{n+1/2} \rangle  &=   -\triangle t \langle \bar{\bCalH}(t_{n+1/2}) \be_{n+1/2} , \be_{n+1/2} \rangle +  \triangle t \langle E_{n+1/2}^{ups} , \be_{n+1/2} \rangle \\
&= \triangle t \langle E_{n+1/2}^{ups} , \be_{n+1/2} \rangle.
\end{align*}
We now use the relation $\be_{n+1/2}  = \dfrac{1}{2} \left( \be_{n+1} + \be_{n} \right)$ and write
\begin{align*}
\langle  \be_{n+1} - \be_{n}  , \dfrac{1}{2} \left( \be_{n+1} + \be_{n} \right) \rangle &=   \dfrac{1}{2}  \langle  \be_{n+1}, \be_{n+1}\rangle  -  \dfrac{1}{2} \langle  \be_{n}, \be_{n}\rangle =\triangle t \langle E_{n+1/2}^{ups} , \be_{n+1/2} \rangle.
\end{align*}
From here and the fact that $\be_0 = 0$, it follows that
\begin{align*}
\langle  \be_{N}, \be_{N} \rangle &= \sum_{n=0}^{N-1}  \left( \langle  \be_{n+1}, \be_{n+1}\rangle - \langle  \be_{n}, \be_{n}\rangle \right) = 2 \triangle t \sum_{n=0}^{N-1}  \langle E_{n+1/2}^{ups} , \be_{n+1/2} \rangle \\
&\leq \triangle t \sum_{n=0}^{N-1} \left|   E_{n+1/2}^{ups}  \right|^2 +  \triangle t \sum_{n=0}^{N-1} \left|   \be_{n+1/2}  \right|^2 \\& \leq  \triangle t \sum_{n=0}^{N-1} \left|   E_{n+1/2}^{ups}  \right|^2   + \triangle t  \sum_{n=0}^{N-1} \left|   \be_{n}  \right|^2 + \dfrac{\triangle t}{2} \left|   \be_{N}  \right|^2. 
\end{align*}
We have with $\triangle t \leq 2/3$,
\begin{align*}
 \left|   \be_{N}  \right|^2 &\leq \left( 1  - \dfrac{\triangle t}{2} \right)^{-1} \left(    N \triangle t  \max_{ 0 \leq n \leq N-1} \left| E_{n+1/2}^{ups} \right|^2  + \triangle t  \sum_{n=0}^{N-1} \left|   \be_{n}  \right|^2  \right) \\
 &\leq  \left( 1  - \dfrac{\triangle t}{2} \right)^{-1}    N \triangle t  \max_{ 0 \leq n \leq N-1} \left| E_{n+1/2}^{ups} \right|^2  + \dfrac{2 \triangle t}{2-\triangle t}  \sum_{n=0}^{N-1} \left|   \be_{n}  \right|^2  \\
 &\leq \dfrac{3}{2} N \triangle t  \max_{ 0 \leq n \leq N-1} \left| E_{n+1/2}^{ups} \right|^2  + \dfrac{3 \triangle t}{2} \sum_{n=0}^{N-1} \left|   \be_{n}  \right|^2.
\end{align*}
By the discrete Gronwall's inequality and using Theorem
\ref{Upscaling_Estimate_Thm}, we obtain
$$
E_{HMM}: = \left| \bM^{\triangle t}_{N}  - \bar{\bm}_N^{\triangle t} \right| \leq C \sqrt{T} \left(  \tau  + \left( \dfrac{\e}{\tau} \right)^{q+2}  \right) e^{\frac{3}{4} T} \max_{0\leq n \leq N-1} | \bM_{n+1/2}|.
$$
\end{proof}
\begin{Remark}
  Note the constraint on the time step, $\triangle t \leq 2/3$, in
  Theorem \ref{Theorem1_HMM_Error}. It appears despite using the
  implicit midpoint rule on the macro level. Nevertheless, the
  restriction on the time step does not influence the efficiency as it
  still allows for large macroscopic stepsizes.
\end{Remark}
We present now a lemma which estimates the length of the macroscopic
magnetisation at different discrete time instants.
\begin{Lemma} \label{Magnetisation_Length_Lemma} Suppose that $\{ \bM_n^{\triangle t} \}_{n=0}^{N}$ solves \eqref{Macro_Solver_NewNotation} with an initial magnetisation that satisfies $| \bM_{0}| = 1$. Moreover, let $K \in \mathbb{K}^{p,q}$ and $\delta = \tau  + \left( \e/\tau \right)^{q+2}$. Then
\begin{equation*}
|  \bM_n  | \leq \exp(C {t_n \frac{\delta}{1- \frac{\delta \triangle t}{2}}})= 1  + O(\delta),
\end{equation*}
where $C$ is a constant independent of $n,\e,\tau$ and $t_n = n \triangle t$.
\end{Lemma}
\begin{proof} We take the inner-product of \eqref{Macro_Solver_NewNotation} with $\bM_{n+1/2}$. Following the same idea as in the proof of Theorem \ref{Theorem1_HMM_Error} and using Theorem \ref{Upscaling_Estimate_Thm} we obtain
\begin{align*}
\langle  \bM_{n+1}  - \bM_{n+1} \rangle -  \langle  \bM_{n}  - \bM_{n} \rangle & = 2 \triangle t \langle E_{n+1/2}^{ups} , \bM_{n+1/2} \rangle \\
& \leq \triangle t | E_{n+1/2}^{ups}| |\bM_{n+1/2}|  \leq C \triangle t   \delta |\bM_{n+1/2}|^2 \\
& \leq C \dfrac{\triangle t \delta}{2} \left( |\bM_{n+1}|^2  + | \bM_{n} |^{2} \right).
\end{align*}
Hence 
$$
|\bM_{n+1}|^{2} \leq \left( \dfrac{1 + C \frac{\delta \triangle t}{2}}{1 - C \frac{\delta \triangle t}{2}} \right) |\bM_{n}|^2.
$$
From here and the fact that $|\bM_0|_2 =1$ it follows that, with $z := x/(1-x)$ where $x = C \frac{\delta \triangle t}{2}$,
$$|\bM_{n}|^{2} \leq \left( \dfrac{1 + C \frac{\delta \triangle t}{2}}{1 - C \frac{\delta \triangle t}{2}} \right)^n  = \left( 1 + 2 z \right)^n \leq \exp(2 n z).$$
The final result follows by putting $t_n = n \triangle t$. 
\end{proof}
We now present a theorem where an estimate of the HMM error, $E_{HMM}$, for the nonlinear case when $\gamma > 0 $ is given.
\begin{Theorem} \label{Theorem2_HMM_Error} Suppose that $\{ \bM_{n}^{\triangle t} \}_{n=0}^{N}$ solves \eqref{Macro_Solver_eqn}  with $\beta = 1$ and $\gamma  > 0$, where the associated micro problem \eqref{Micro_Problem_eqn} is solved exactly. Moreover, assume $K \in \mathbb{K}^{p,q}$ and that $\bar{\bm}_{n+1}^{\triangle t}$ solves \eqref{Discrete_Averaged_Eqn} and $\triangle t \leq \frac{2}{3 \kappa}$ with $\kappa  = (1 + 2 c_1 \gamma  + 2 C_1 \gamma)$, where $C_1 = O(1)$ is a constant such that $|\bM_{n+1/2}^{\triangle t}| \leq C_1$, and
$$
\langle \partial_\bM \bN_{\gamma}(\bar{\bm}_{n+1/2} + \theta \be_{n+1/2}) \be_{n+1/2}, \be_{n+1/2} \rangle \leq \gamma c_1 \left| \be_{n+1/2} \right|^2, \quad \forall  \quad \theta  \in [0,1].
$$
Then the error $\be_n^{\triangle t} = \bM_{n}^{\triangle t} - \bar{\bm}_{n}^{\triangle t}$ satisfies
 \begin{equation*}
\left| \be_N^{\triangle t} \right|  \leq C \sqrt{T} \left(  \tau  + \left( \dfrac{\e}{\tau} \right)^{q+2}  \right) e^{\frac{3 \kappa}{4} T} \max_{0 \leq n \leq N} \left| \bM_{n+1/2}^{\triangle t} \right|,
\end{equation*}
where $T = N \triangle t$, and $C$ is a constant independent of $\e, \tau$ and $T$ but may depend on $\bCalH,p ,q , K$ or $\gamma$.
\end{Theorem}
\begin{proof}
For convenience we write $-\gamma \mathcal{J}$ to denote the Jacobian $\partial_\bM \bN_{\gamma}$ and we remove $\triangle t$ in the notations.  By \eqref{Discrete_Averaged_Eqn} and \eqref{Macro_Solver_NewNotation}, we see that the error $\be_n$ satisfies
\begin{align*}
\be_{n+1} &= \be_{n} - \triangle t \bar{\bCalH}(t_{n+1/2}) \be_{n+1/2} + \triangle t E_{n+1/2}^{ups}  + \triangle t \left(   \bN_{\gamma}[\bM_{n+1/2}] -   \bN_{\gamma}[\bar{\bm}_{n+1/2}]  \right) \\ &- \gamma \triangle t \bM_{n+1/2}\times E^{ups}_{n+1/2} \\
&= \be_{n} - \triangle t \bar{\bCalH}(t_{n+1/2}) \be_{n+1/2} + \triangle t E_{n+1/2}^{ups}  - \gamma \triangle t \mathcal{J}[\bar{\bm}_{n+1/2} +  \theta \be_{n+1/2}] \be_{n+1/2} \\&- \gamma \triangle t \bM_{n+1/2}\times E^{ups}_{n+1/2}.
\end{align*}
Hence
\begin{align*}
\langle  \be_{n+1} -\be_{n}, \be_{n+1} \rangle  &= \triangle t \langle  E_{n+1/2}^{ups} , \be_{n+1/2}\rangle - \gamma \triangle t  \langle \mathcal{J}[\bar{\bm}_{n+1/2} +  \theta \be_{n+1/2}] \be_{n+1/2}, \be_{n+1/2}  \rangle \\
&-\gamma \triangle t \langle  \bM_{n+1/2}\times E^{ups}_{n+1/2} , \be_{n+1/2}  \rangle.
\end{align*}
We now use the fact that by Lemma \ref{Magnetisation_Length_Lemma}, $|\bM_{n+1/2}| \leq C_1$, where $C_1 = O(1)$ for sufficiently small values for $\delta$ and we proceed similar to the analysis of the linear case. Moreover, we define $\kappa  := \left(  1 + 2 c_1 \gamma  + 2 C_1 \gamma  \right)$ and find that
\begin{align*}
\langle  \be_{N}, \be_{N} \rangle &= \sum_{n=0}^{N-1}  \left( \langle  \be_{n+1}, \be_{n+1}\rangle - \langle  \be_{n}, \be_{n}\rangle \right) \\ & \hspace{-1cm}= 2 \triangle t \sum_{n=0}^{N-1}  \langle E_{n+1/2}^{ups} , \be_{n+1/2} \rangle 
- 2\gamma \triangle t \sum_{n=0}^{N-1}  \langle \mathcal{J}[\bar{\bm}_{n+1/2} +  \theta \be_{n+1/2}] \be_{n+1/2}, \be_{n+1/2}  \rangle \\ & \hspace{-1cm}- 2 \gamma \triangle t  \sum_{n=0}^{N-1} \langle  \bM_{n+1/2}\times E^{ups}_{n+1/2} , \be_{n+1/2}  \rangle \\
& \hspace{-1cm}\leq  \triangle t \sum_{n=0}^{N-1} \left|   E_{n+1/2}^{ups}  \right|^2 + \left(  1 + 2 c_1 \gamma   \right) \triangle t  \sum_{n=0}^{N-1} \left|   \be_{n+1/2}  \right|^2 + 2 C_1 \gamma \triangle t \sum_{n=0}^{N-1} \left|   E_{n+1/2}^{ups} \right|^2 \\&+ \left| \be_{n+1/2} \right|^2
\leq  (1+ 2 C_1 \gamma) \triangle t \sum_{n=0}^{N-1} \left|   E_{n+1/2}^{ups}  \right|^2   +\kappa \triangle t  \sum_{n=0}^{N-1} \left|   \be_{n}  \right|^2 + \kappa \dfrac{\triangle t}{2} \left|   \be_{N}  \right|^2. 
\end{align*}
Now let $\triangle t \leq \frac{2}{3 \kappa}$. Then
\begin{align*}
\left| \be_N \right|^2 & \leq \dfrac{1 + 2 C_1 \gamma }{1 - \kappa \triangle t /2}   N \triangle t  \max_{0\leq n \leq N} \left| E_{n+1/2}^{ups} \right|^2 +  \dfrac{3 \kappa \triangle t}{2} \sum_{n=0}^{N-1} \left|   \be_{n}  \right|^2 \\ 
&\leq   \dfrac{3 \kappa }{2}   T  \max_{0\leq n \leq N} \left| E_{n+1/2}^{ups} \right|^2 + \dfrac{3 \kappa \triangle t}{2}  \sum_{n=0}^{N-1} \left|   \be_{n}  \right|^2.
 \end{align*}
 The final estimate is obtained by an application of the discrete Gronwall's inequality and Theorem
 \ref{Upscaling_Estimate_Thm}.
\end{proof}
 
 \begin{Remark}\label{Magnetisation_Amplitude}
Using the error estimates in Theorems \ref{Theorem1_HMM_Error} and \ref{Theorem2_HMM_Error}, the estimate in Lemma \ref{Magnetisation_Length_Lemma} and the fact that $|\bar{\bm}_{n}^{\triangle t}|_2 = 1$, one can see that
\begin{align*}
\left|| \bM_{N}^{\triangle t, \delta t} |  - 1 \right| &\leq \left|   \bM_{N}^{\triangle t, \delta t} -  \bM_{N}^{\triangle t} \right|  + \left| \bar{\bm}_N^{\triangle t}  -  \bM_{N}^{\triangle t}\right|  + \Bigl|\underbrace{\left| \bar{\bm}_N^{\triangle t} \right| }_{=1}  -1 \Bigr| \\
&\leq E_{micro}  +  C_0 \sqrt{T}  \left(  \tau  + \left( \dfrac{\e}{\tau} \right)^{q+2}  \right) e^{C_1 T},
\end{align*}
which shows that the magnetisation length does not remain constant for
HMM type multiscale couplings but that the error can be reduced down
to $O(\e)$ by using high order methods at the micro level and by
choosing $\tau \approx \e^{1 - \beta}$, for $\beta = 1/(q+2)$. Finally, a typical estimate for the micro error takes the form $E_{micro} \leq C \left(  \dfrac{\delta t}{\e}  \right)^r$, where $\delta t$ is a microscopic step-size and $r$ represents the order of the accuracy of the micro-solver.
 \end{Remark}

\vspace{2cm}

\section{Numerical results}
\subsection{The tail error}
In this section, we aim at illustrating the validity of the expansion made in the proof of Theorem \ref{Upscaling_Estimate_Thm}. For this we solve the problem \eqref{Scaled_m_eqn} between $0\leq t \leq 1$ with the initial condition $\bM = [0,0,1]^T$, and the external field 
$$\bH(t,z) = [0,0,1]^T + [\sin(2 \pi z),\cos(2 \pi z),0]^T.$$
We study the convergence for two cases: (i) The convergence against $\bm_0$, where $\bm_0 = \bM$, (ii) The convergence against the truncated solution $\tilde{\bm}(t;\e) = \bm_0  + \e \bm_1(t)$, where $\bm_1$ solves \eqref{m1_eqn}. The error is defined as 
$$
E_{tail} = \max_{0\leq t \leq 1} \left|  \bm(t;\e)  - {\bf{u}}(t;\e) \right|_2, \quad \text{ where }  {\bf{u}}(t;\e) = \bm_0 \text{ or } \tilde{\bm}(t;\e).
$$
The left plot in Figure \ref{Fig_TailFConv} shows that we gain one order in $\e$ upon including the term $\bm_1$ in the expansion and, in particular, that the estimate \eqref{e_Estimate_eqn} holds.
\subsection{The upscaling error}
We show now evidence that the estimate, for the upscaling error, given in Theorem \ref{Upscaling_Estimate_Thm}, is sharp. For this we compare the upscaled flux in \eqref{Upscaling_eqn} with the exact effective flux $\bar{\bCalH}(t) \bM$ for $\bM = [0,0,1]^T$. The micro problem is solved with the initial data $ [0,0,1]^T$ and the external field 
\begin{equation}\label{External_Field_eqn}\bH^{\e}(t) = [0,0,1]^T + [\cos(2 \pi t/\e),\cos(2 \pi t/\e),0]^T.\end{equation}
In this case $\bar{\mathcal{H}}  = [0,0,1]^T$. The convergence is shown for two cases: (i) By setting $\tau = 5.3 \e$, and letting $\e \longrightarrow 0$ for increasing values of $q$, where higher $q$ means a better regularity for the kernel $K$, see the right plot in Figure \ref{Fig_TailFConv} (ii) By fixing $\tau = 0.1$ and $\tau = 0.2$ and study the difference as $\e \longrightarrow 0$, see the Figure \ref{Fig_FConv_FixedTau}. Both tests corroborate our theoretical findings in Theorem \ref{Upscaling_Estimate_Thm}.
\subsection{The full solution}
In Figure \ref{Fig_HMMvsFullSolutionDet} (top plots), we consider two examples with 
$$ \bH(t,t/\e) = \bH_0  + \bH_1(t/\e),
$$
where $\bH_1(t)  = [\sin(2 \pi t), \cos(2 \pi t),0]^T$ or $\bH_1(t)  = [\sin(2 \pi t)^2, \cos(2 \pi t)^2,0]^T$ and $\bH_0 = [0,0,1]^T$ in both cases. We observe that, in the former, the solution $\bm^{\e}$ converges to $[0,0,1]^T$ as $T$ increases whereas in the latter, the solution converges to a non-trivial limit given by the average. In both cases, the HMM solution captures the coarse features of the exact solution using only $20$ discretisation points. In the same Figure (Figure \ref{Fig_HMMvsFullSolutionDet} bottom plot), we depict a numerical solution with a damping parameter $\gamma = 0.1$, where all other physical and numerical parameters are chosen the same as in the top plots. The HMM solution again captures the correct magnetisation dynamics.


\subsection{Magnetisation amplitude} We consider the full problem with an external field of the form \eqref{External_Field_eqn}, where $\bH_1(t)  = [\sin(2 \pi t), \cos(2 \pi t),0]^T$ and $\bH_0 = [0,0,1]^T$ and we look at two scenarios: (i) For a fixed $\e$ we keep track of the evolution of the magnetisation amplitude over the macro time steps, (ii) We study the deviation of the magnetisation amplitude in comparison to the initial length of magnetisation as we refine the small scale parameter $\varepsilon$. The initial magnetisation is given as $\bM_0  = \frac{1}{\sqrt{3}}[1,1,1]^T$ so that $|\bM_0| = 1$. Figure \ref{Fig_MagnetisationLength} shows that the HMM does not preserve the initial length of the magnetisation. In the left plot, we show the evolution of the magnetisation amplitude for each macroscopic time when $\e = 0.01$. In this simulation we have used $N \triangle t = 2 \pi$ and $N=20$. The magnetisation length $|\bM_n|$ for $n>0$ clearly deviates from $|\bM_0| = 1$. In the right plot, however, we study the convergence 
$$
\max_{0 \leq n \leq N} \left| \left|  \bM_n \right|_2  - \left|  \bM_0 \right|_2 \right| \longrightarrow 0
$$
as $\e \longrightarrow 0$, for $N \triangle t  = 1$ and $N=10$. The simulation result verifies the expected convergence rate stated in Remark \ref{Magnetisation_Amplitude}. In the simulations we have chosen a kernel $K \in \mathbb{K}^{p,q}$ with $p=1$ and $q=7$. Moreover, we set $\tau  = 5.3 \e$ and observe $O(\e)$ convergence  in the selected range of $\e$. Note that the $O((\e/\tau)^{q+2})$ error term is not seen in this simulation since we have used a large $q$. 

 \begin{figure}[h] 
    \centering
        \includegraphics[width=0.48\textwidth]{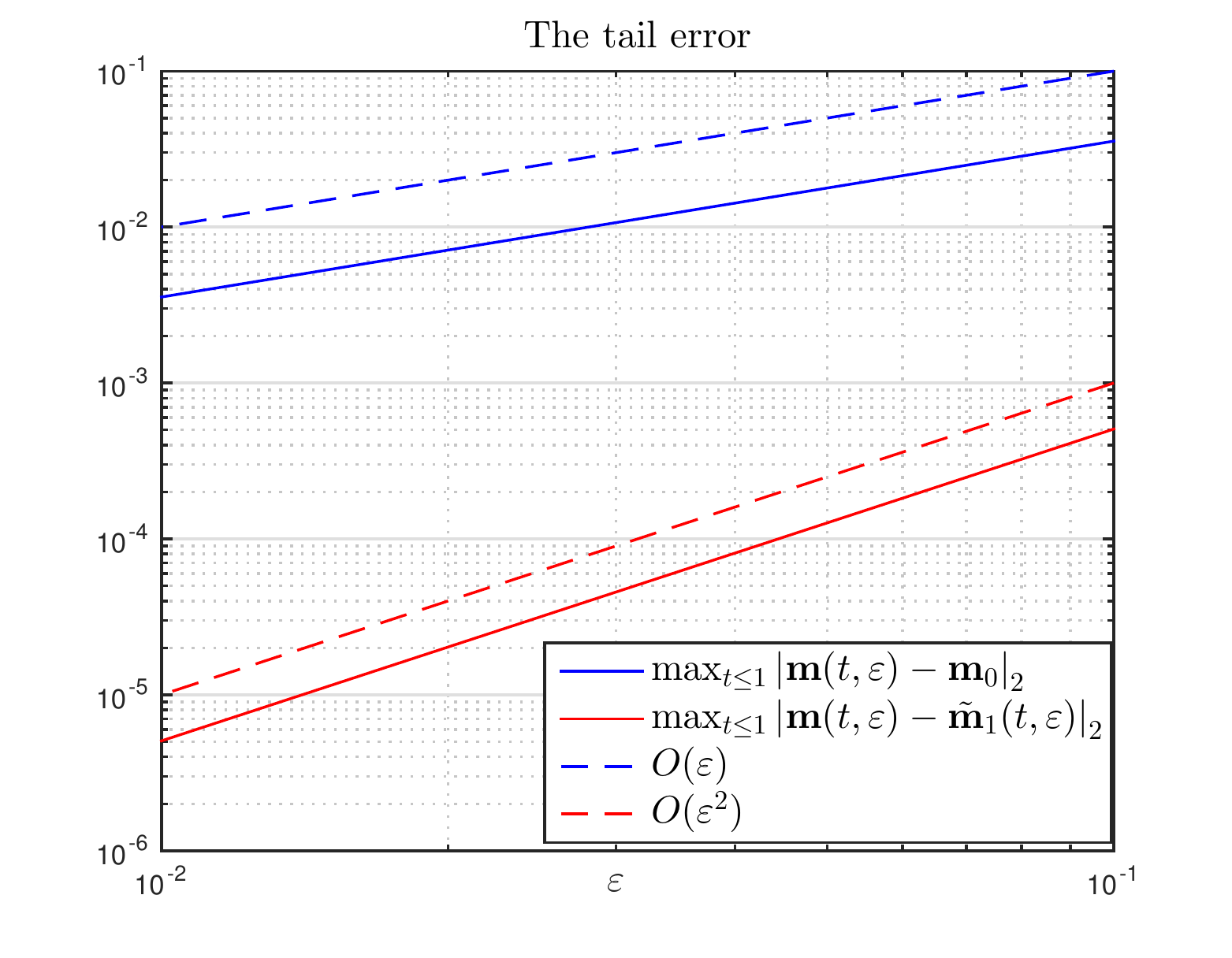}
        \includegraphics[width=0.48\textwidth]{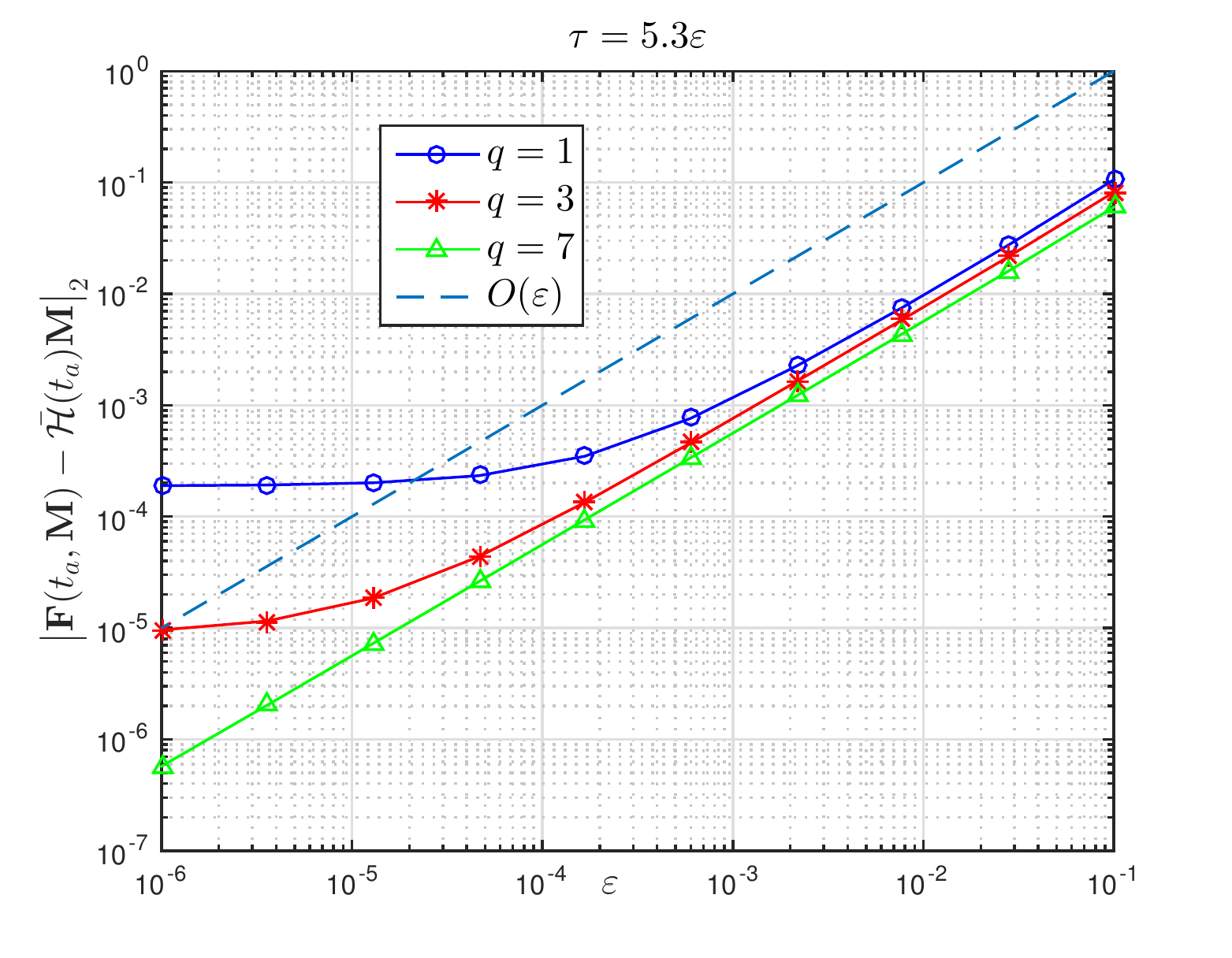}
    \caption{(Left) The differences $\bm - \bm_0$ and $\bm - \tilde{\bm}$, where $\tilde{\bm} = \bm_0 + \e \bm_1$ as $\e \longrightarrow 0$ are shown.  The result shows that the estimate \eqref{e_Estimate_eqn} is sharp.  (Right) The upscaling error, which is estimated in Theorem \ref{Upscaling_Estimate_Thm}, is depicted. An averaging domain of size $\tau = 5.3 \e$ is used in the simulation. The constant part of the error corresponds to $O((\e/\tau)^{q+2})$ from Theorem \ref{Upscaling_Estimate_Thm} and it decreases upon using smoother kernels with high $q$. This result confirms the error estimate from the Theorem.}
    \label{Fig_TailFConv}
\end{figure}

\begin{figure}[h]
    \centering
       \includegraphics[width=0.48\textwidth]{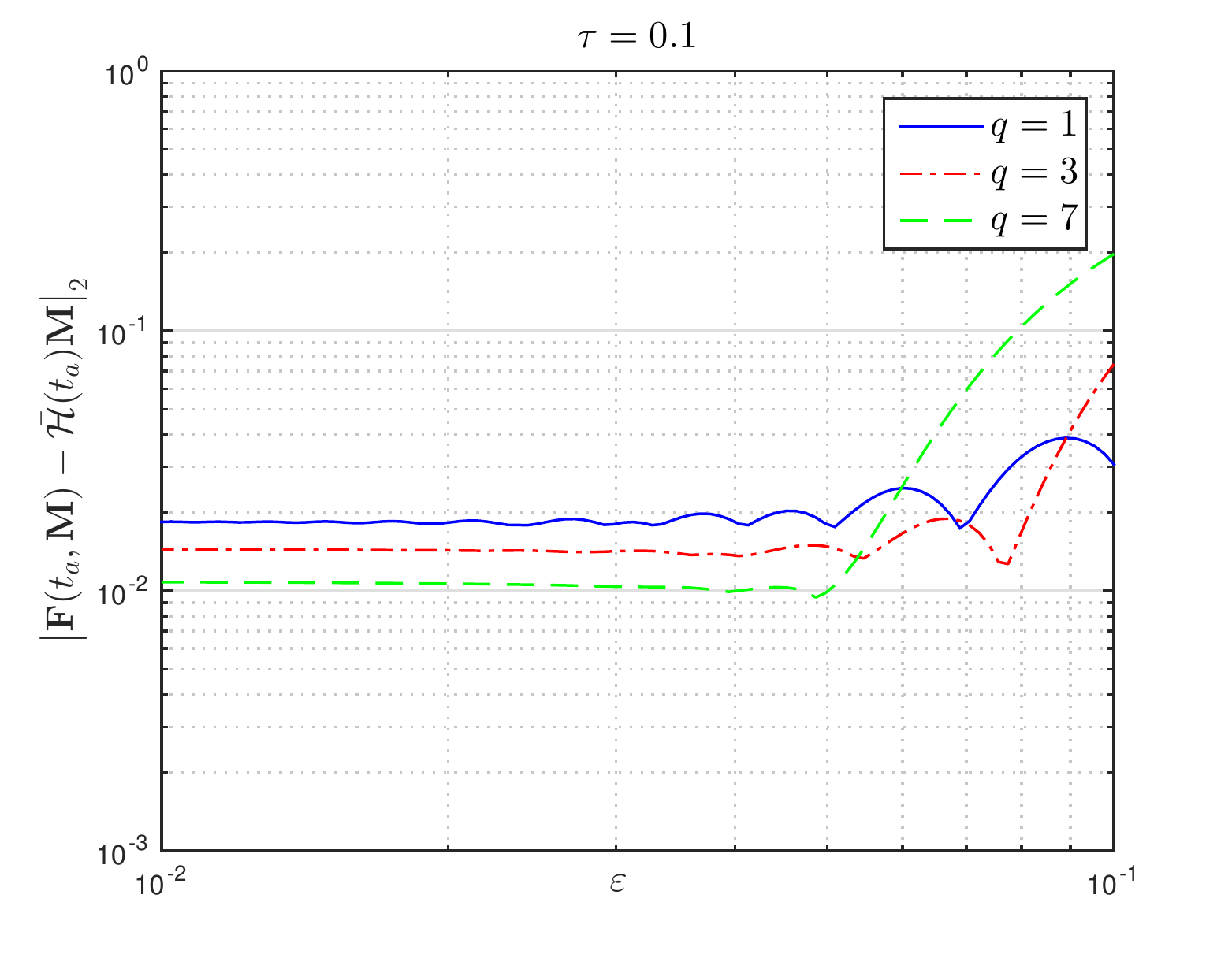}
        \label{fig:gull}
        \includegraphics[width=0.48\textwidth]{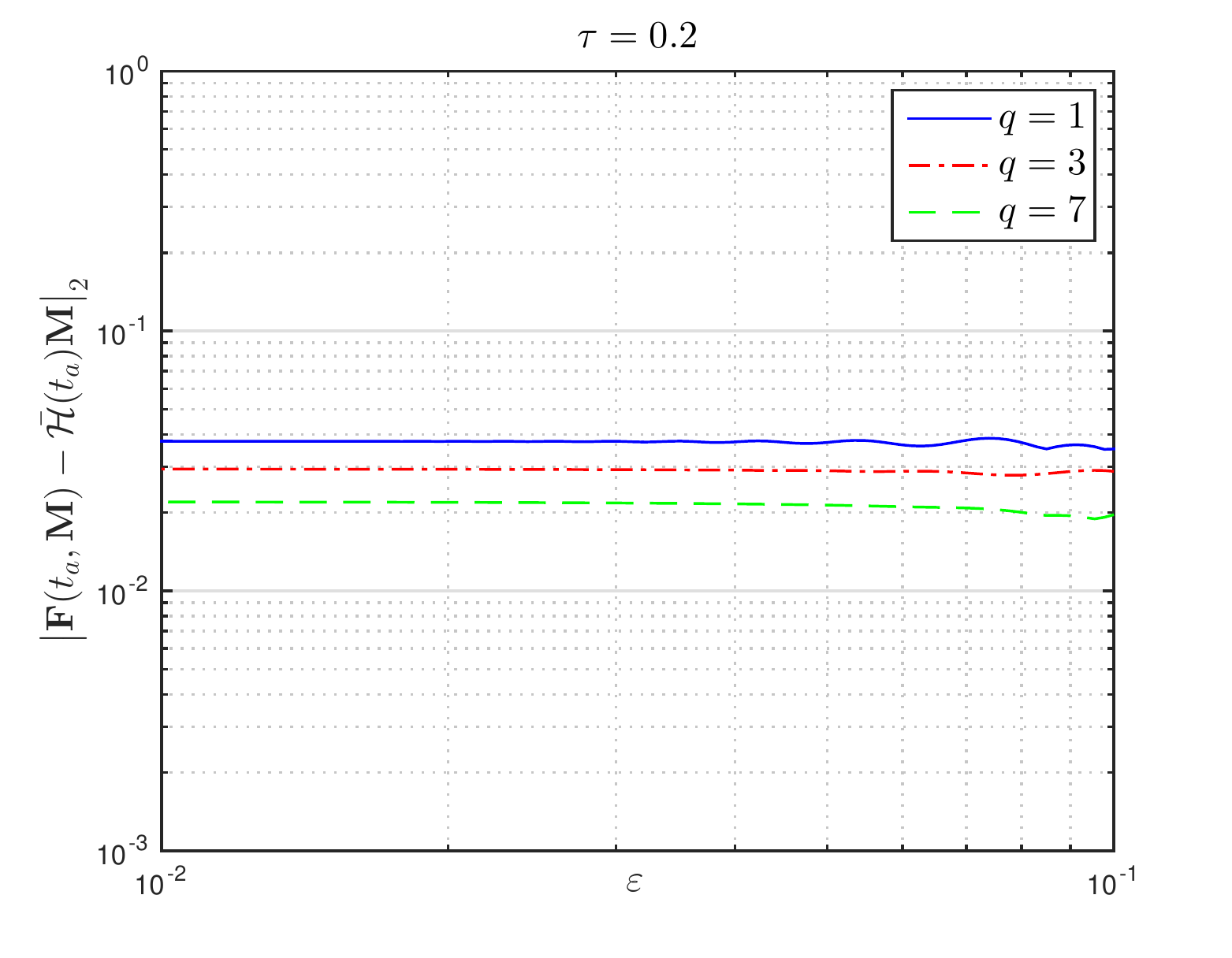}
    \caption{The upscaling error from Theorem \ref{Upscaling_Estimate_Thm} is shown. The error, $|\bF - \bar{\bCalH}(t_a) \bM |_2$, as $\e \longrightarrow 0$ is depicted for (Left)  $\tau = 0.1$, (Right) $\tau = 0.2$. The results verify the estimate in the theorem.}
    \label{Fig_FConv_FixedTau}
\end{figure}

\begin{figure}[h]
  \centering
    \includegraphics[width=0.48\textwidth]{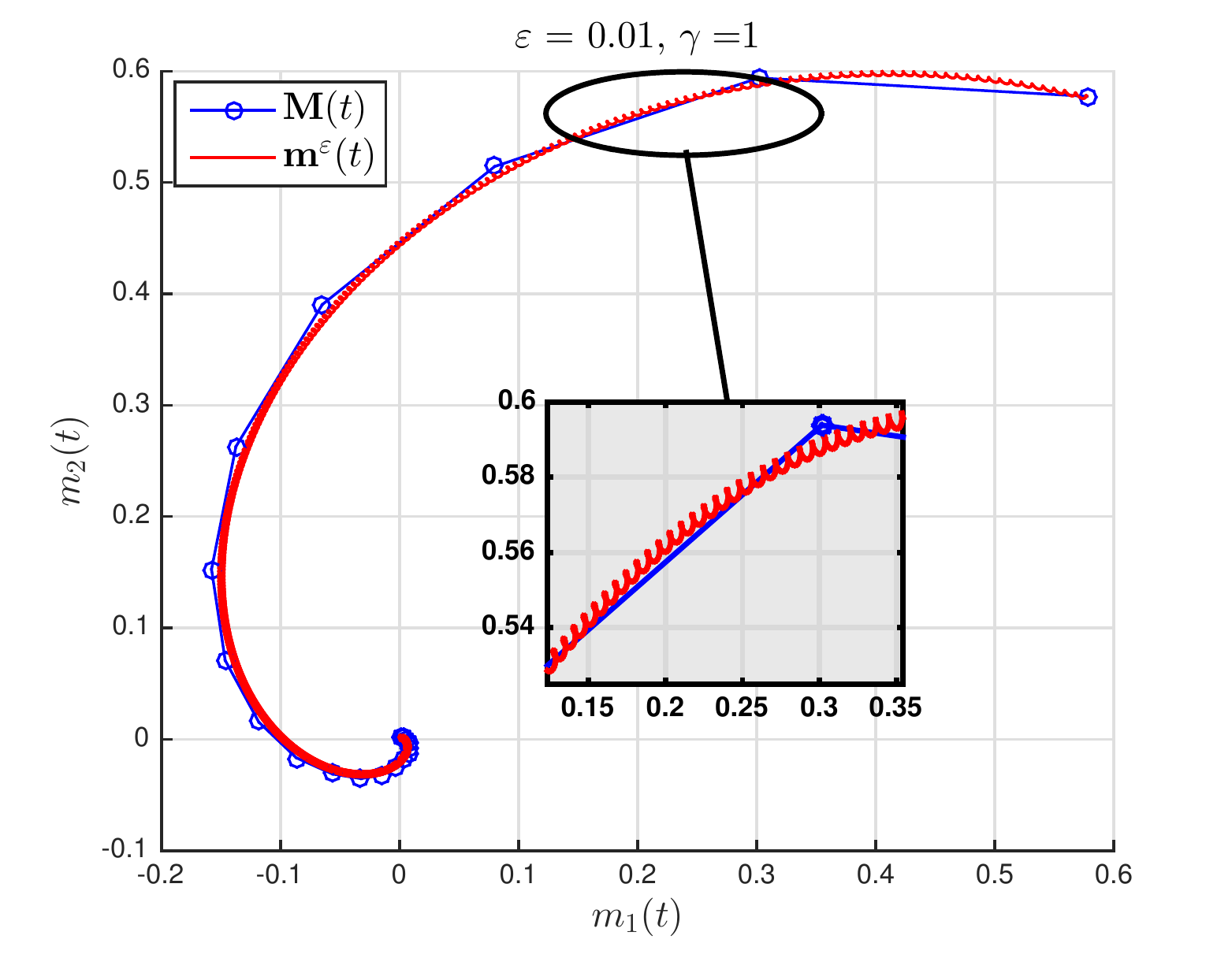}
    \includegraphics[width=0.48\textwidth]{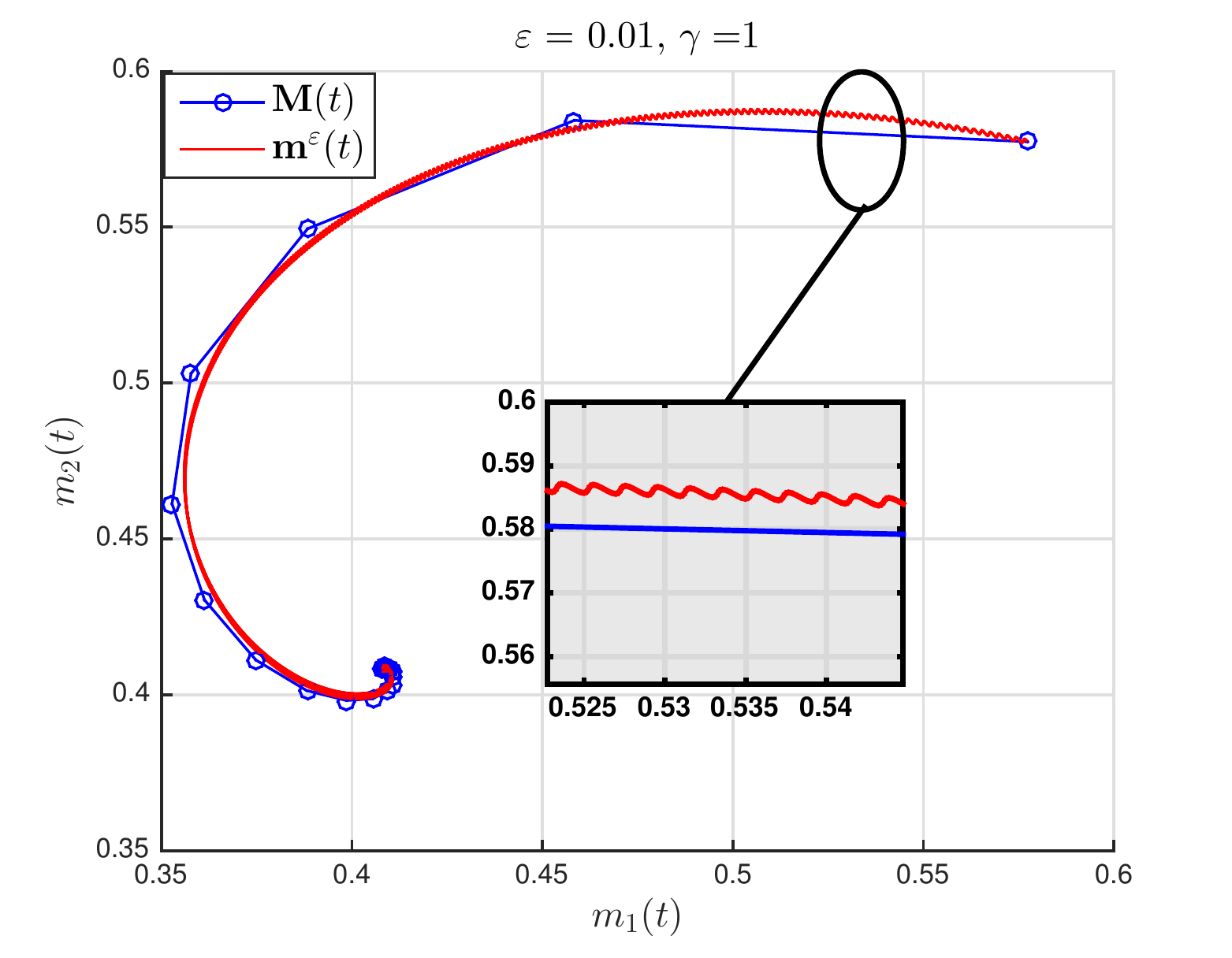}
      \includegraphics[width=0.48\textwidth]{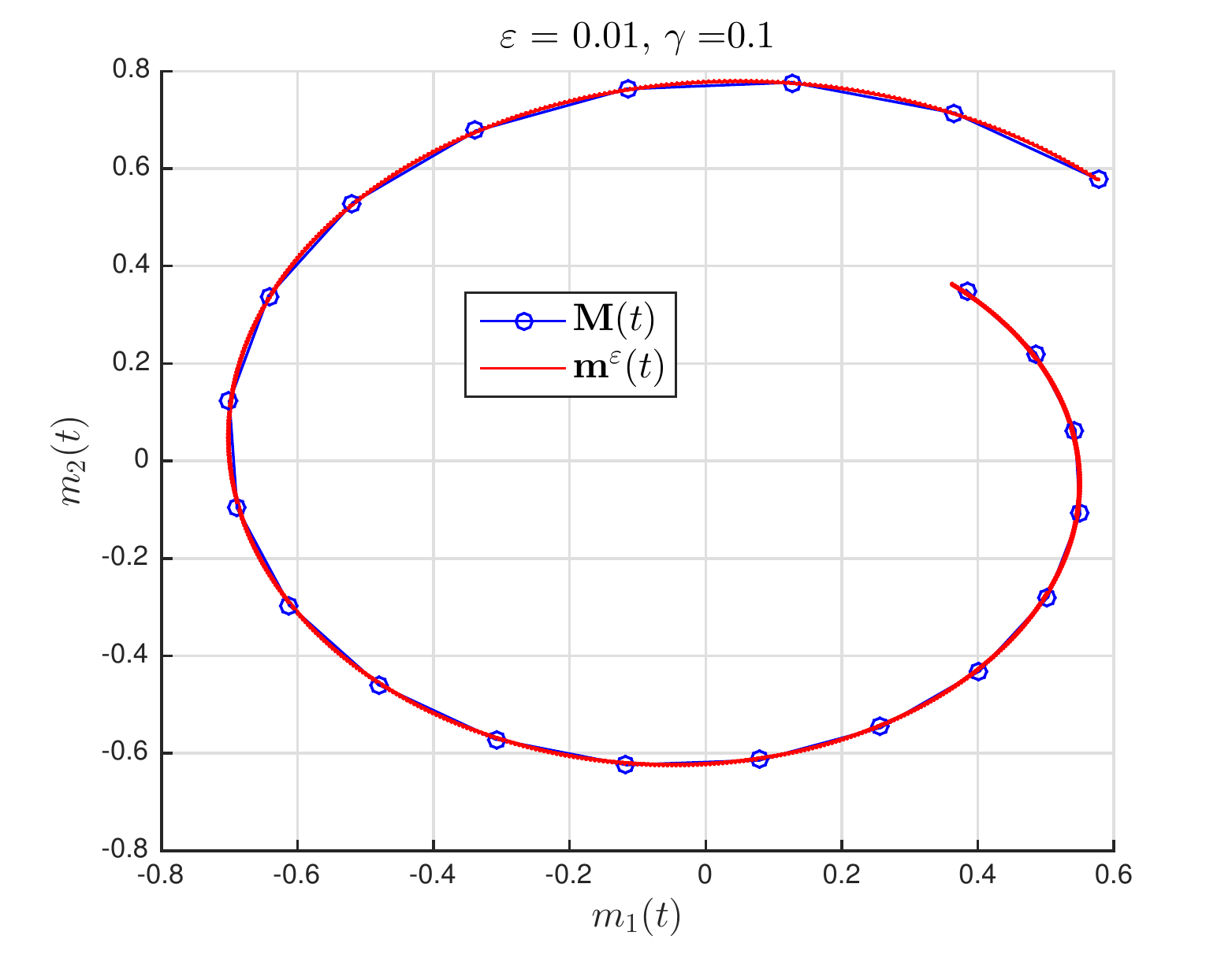}

  \caption{The HMM solution $\bM$ with $\bH(t,t/\e) = \bH_0 + \bH_1(t/\e)$, where $\bH_0 = [0,0,1]^T$ and (Left top) $\bH_1(t/\e) = [\sin(2 \pi t/\e), \cos(2 \pi t/\e),0]^T$ (Right top) $\bH_1(t/\e) = [\sin(2 \pi t/\e)^2, \cos(2 \pi t/\e)^2,0]^T$, for $\e = 0.01$ is compared to the exact solution $\bm^{\e}$. We have used $\beta = 1$, $\gamma=1$, $\tau = 5 \e$, $T = 2 \pi$, $\triangle t = T/20$ and a kernel $K \in \mathbb{K}^{p,q}$ with $p=5, q=4$ in the simulations. The average behaviour of the fine scale solution is recovered by using only $20$ points on the macroscopic grid. (Bottom) The HMM solution $\bM$ with $\bH_1(t/\e) = [\sin(2 \pi t/\e), \cos(2 \pi t/\e),0]^T$, and $\gamma=0.1$. All other parameters are chosen the same as in the top simulations.}
    \label{Fig_HMMvsFullSolutionDet}
\end{figure}

\begin{figure}[h]
    \centering
      \includegraphics[width=0.48\textwidth]{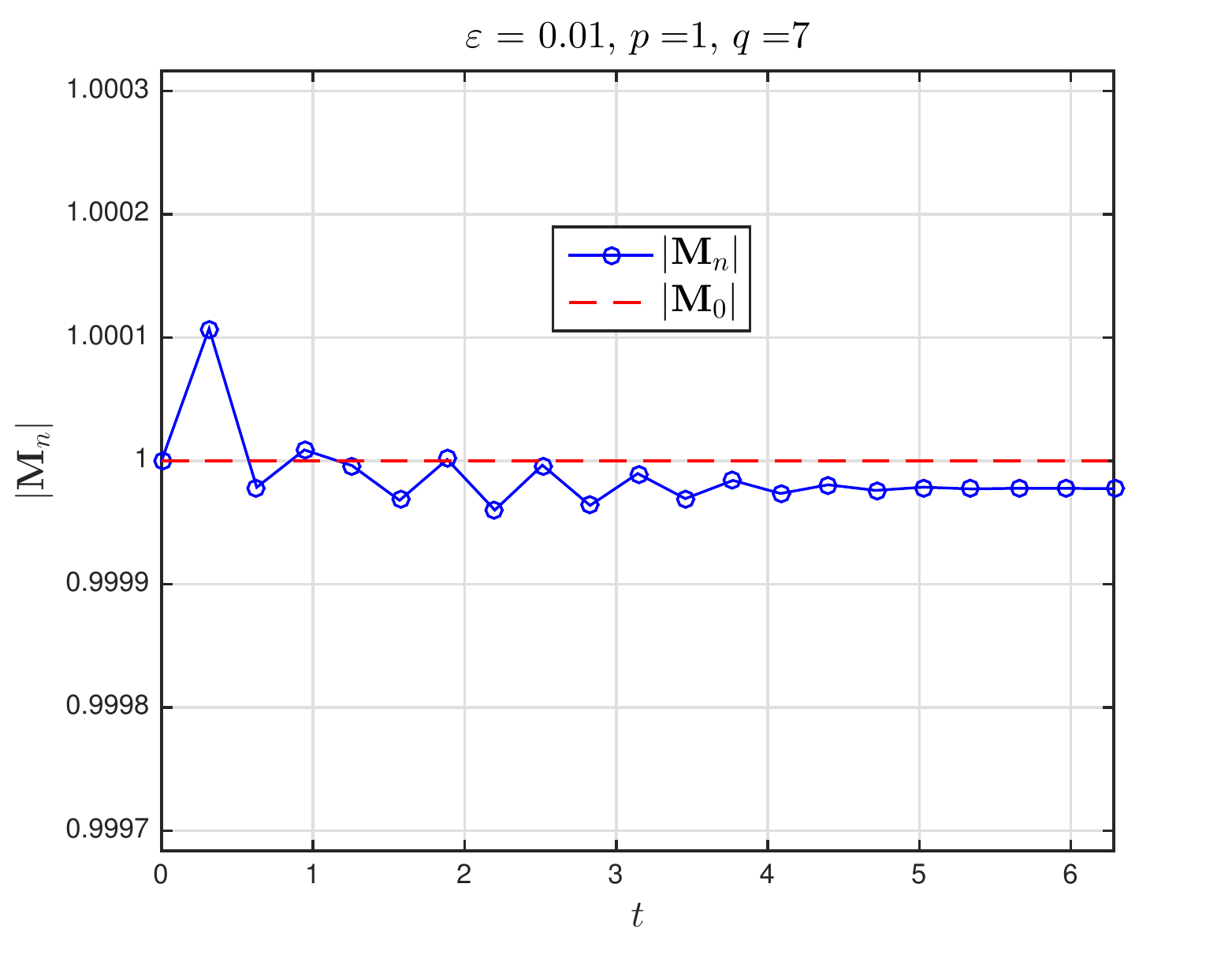}
        \includegraphics[width=0.48\textwidth]{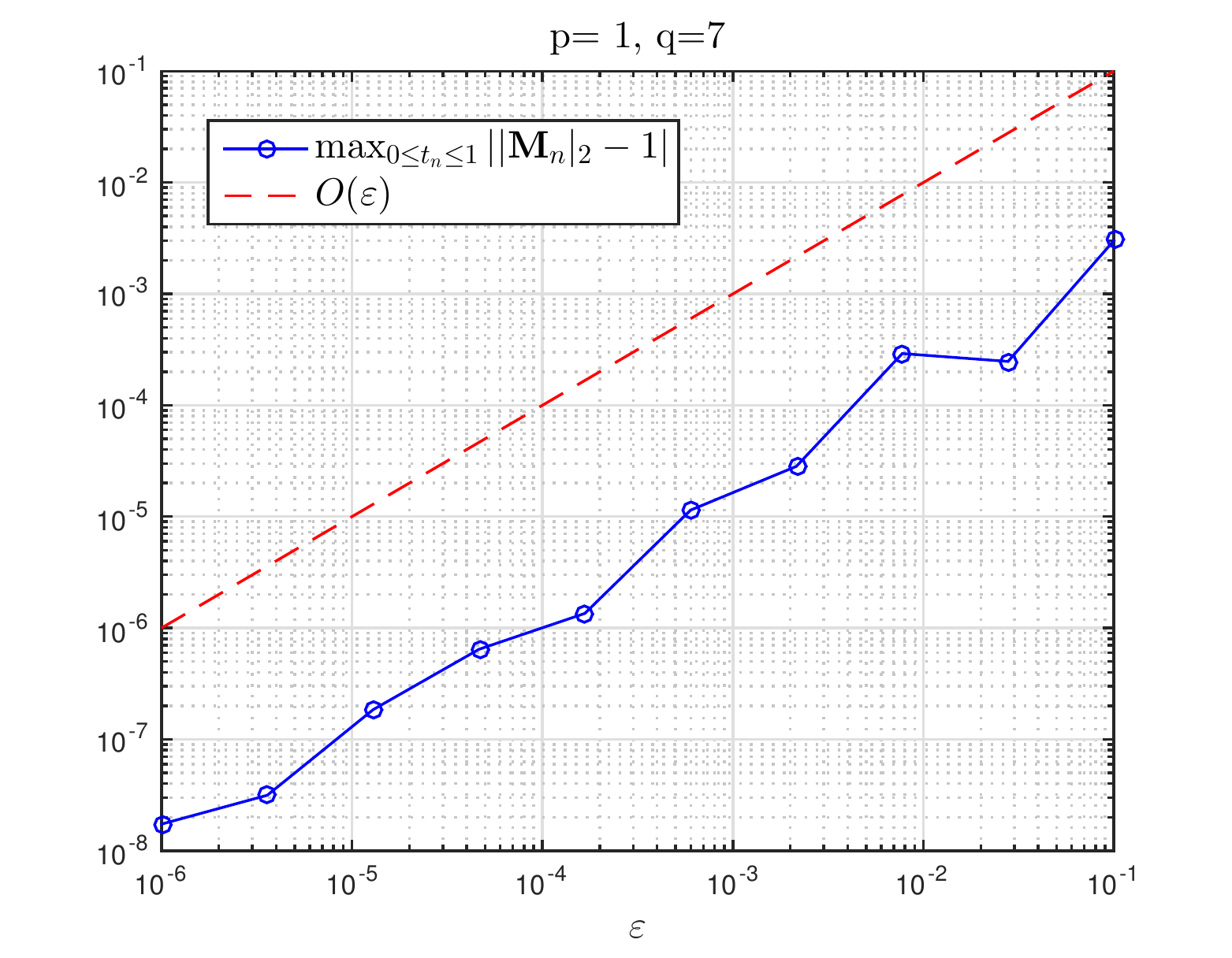}
     \caption{The magnetisation amplitude $|\bM_n|_{n=0}^{N}$ with $H(t,t/\e) = \bH_0  + \bH_1(t/\e)$, where $\bH_0 = [0,0,1]^T$ and $\bH_1(t) = [\sin(2\pi t), \cos(2 \pi t), 0]^T$ is computed using a kernel $K \in \mathbb{K}^{p,q}$ with $p=1$ and $q=7$. (Left) The evolution of the magnetisation amplitude is illustrated for $\e = 0.01$. The HMM does not conserve the initial length. (Right)  The convergence of the magnetisation length to unit length, as $\e \longrightarrow 0$, for $\tau = 5.3 \e$ is depicted. The $O(\e)$ convergence rate is in a good agreement with the estimate provided in Remark \ref{Magnetisation_Amplitude}. The error $O((\e/\tau)^{q+2})$ is not seen in this plot as $q$ is chosen sufficiently large. }
        \label{Fig_MagnetisationLength}
\end{figure}

\subsection{Numerical results for a chain of magnetic particles}
In this subsection, the algorithm described in Subsection \ref{Sec:HMMChain} is applied to a test case. The full microscopic problem contains $N=100$ particles, with the initial configuration
$$
\bm^{\e}_{i}  = [\sin(2\pi x_i), \cos(2 \pi x_i),0]^T, \quad i = 1, \ldots, N,
$$
where $x_i = i \delta x$, and $\delta x = 0.01$. An external field of the form 
$$
\tilde{\bH}^{\e}(t) = (1 + \cos(0.43 t) + \cos(2\pi t/\e)^2) [0,0,1]^T
$$
is applied to all the particles. The physical parameters are chosen as $\beta = 1$  and $\gamma = 1$. The HMM solution is computed with a macroscopic timestep $\triangle t = T/20$, for a final simulation time $T=1.5$, and a spatial stepsize $\triangle x = 0.1$ is used over the macroscopic simulation box $[0,1]$. The micro problem in the HMM uses $11$ particles ($r = 5$), a microscopic time of the size $\tau = 5 \e$, where $\e = 0.01$, and a kernel $K$ with $p=5$ and $q=4$. Figure \ref{Fig_ChainSimulations} shows the evolution of the HMM solution versus a direct numerical simulation (DNS) from $t=0$ to $t=T$ for four different time instants. Both solutions converge to the right steady state value (aligned in the direction of the external field). Moreover, the HMM solution captures the correct dynamics using fewer points in time and space. 

\begin{figure}[h]
    \centering
      \includegraphics[width=0.48\textwidth]{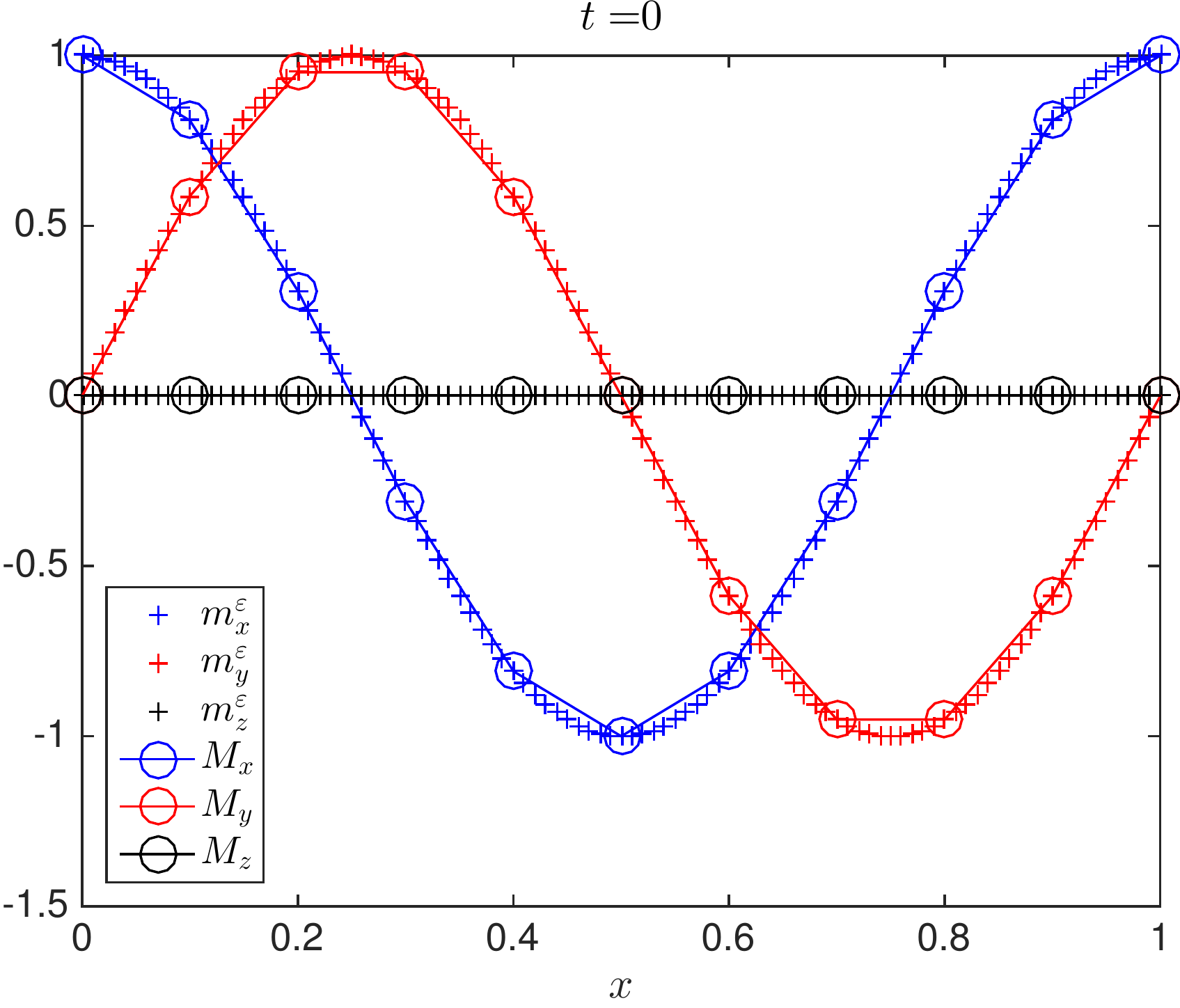}
        \includegraphics[width=0.48\textwidth]{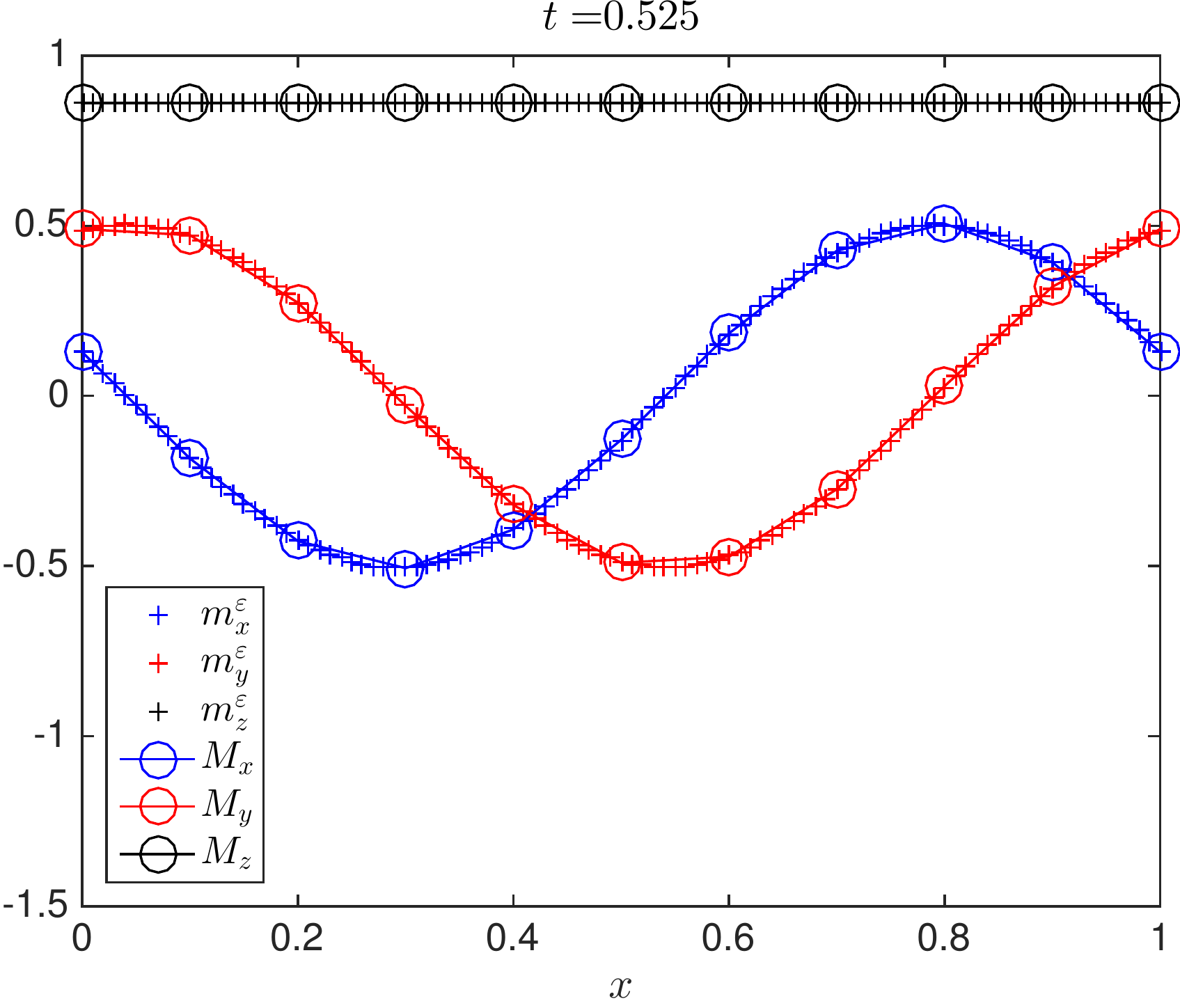}
                \includegraphics[width=0.48\textwidth]{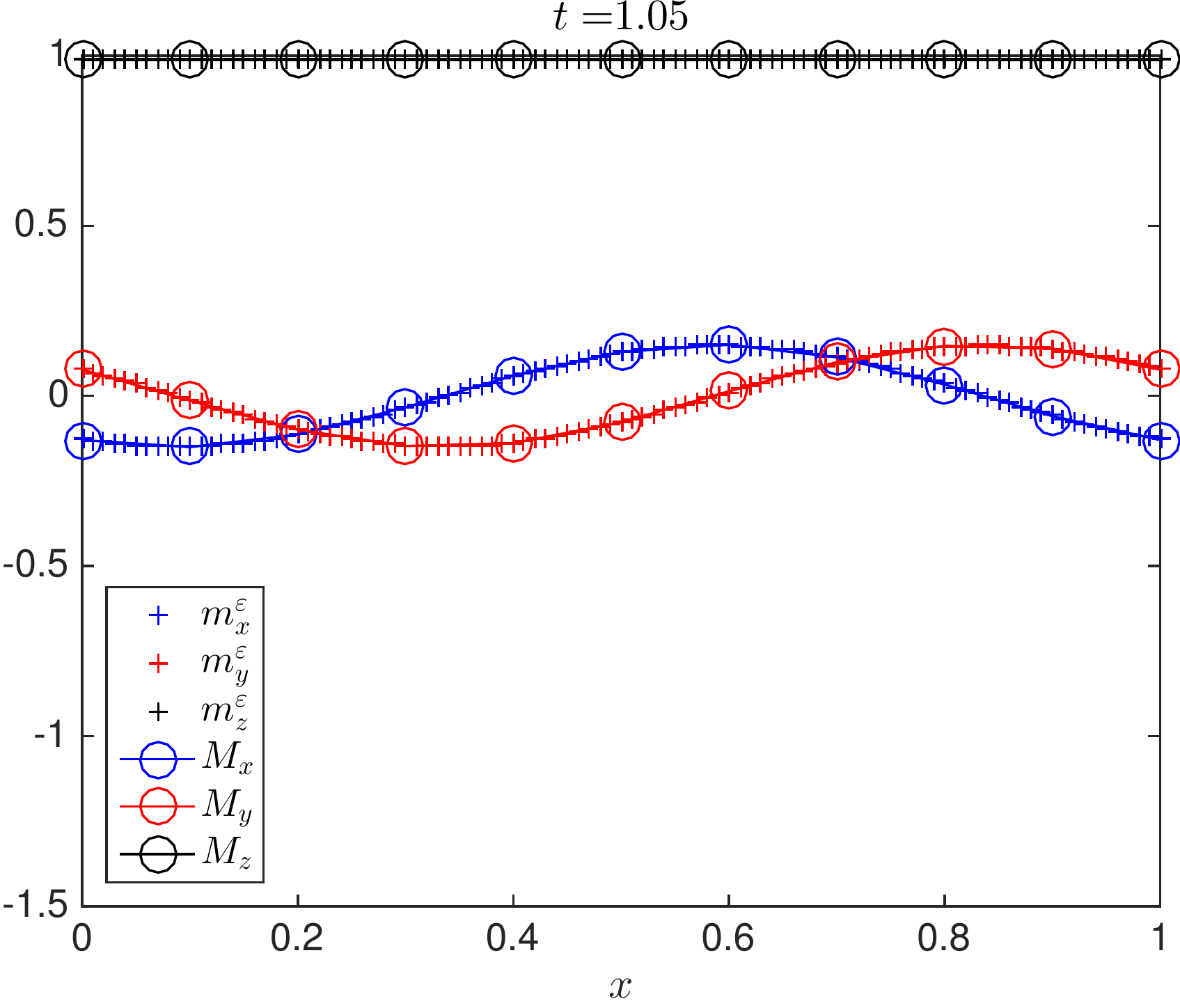}
        \includegraphics[width=0.48\textwidth]{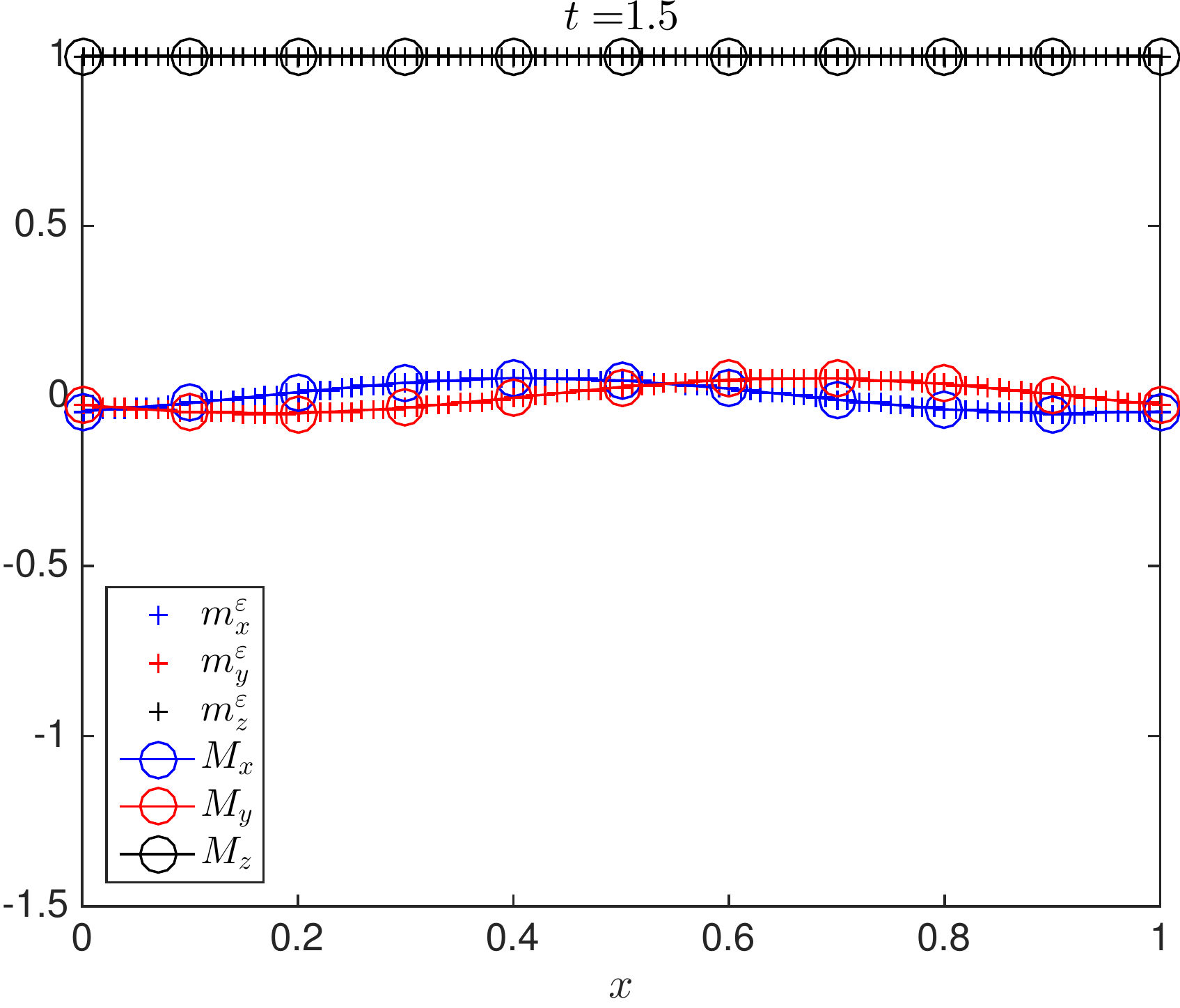}
     \caption{HMM solutions (in circles) with $10$ macroscopic points are compared to a direct numerical simulation with $100$ particles. At $t=0$, the particles are aligned non-uniformly in the $x$ and $y$ directions. An external field is applied in the $z$ direction. As time evolves, both the HMM and DNS solutions converge to the steady state, which is in the $z$ direction. The HMM is observed to capture the correct dynamics using fewer points in time and space.}
        \label{Fig_ChainSimulations}
\end{figure}

\clearpage


\section{Discussion}
We have proposed and analysed a multiscale method based on the HMM framework for coupling disparate scales in the Landau-Lifshitz equations of micromagnetism. The current study serves as a first step in designing general upscaling algorithms for incorporating effects of microscopic variations into smoothly varying macroscopic dynamics. 

For the analysis, we have considered a simple model problem where the microscopic model is driven by a high frequency field. We have shown that the HMM captures the right effective quantities by proving an error estimate between the HMM and the exact averaged dynamics.  Two important outcomes of the analysis are: (i) The error between the HMM and the exact dynamics consist of discretisation errors in the micro and macro levels and the upscaling error which can be pushed down to $O(\e)$, where $\e \ll 1$ is a typical lengthscale representing the ratio between the fine and the coarse scales, (ii) Unlike one-scale methods such as the implicit mid-point rule, the multiscale algorithms would, in general, suffer from non-constant length of the magnetisation. However this error can be controlled upon choosing an appropriate high-order micro solver along with a suitable choice for the size of the local averaging domain, see Remark \ref{Magnetisation_Amplitude}. Alternatively, a projection step could be added at each time iteration to project the magnetisation to the unit sphere, which is a common practice in micro-magnetic simulations, see e.g. \cite{E_Wang,Lewis_Nigam}. The extension of the method to a chain of magnetic particles is also presented. The algorithm uses upscaling both in time and space, and it is shown  (via a numerical example) that the method captures the evolution of the correct magnetisation dynamics. 

We believe that the results in the present article give a good insight
regarding the nature of the errors that may appear in more general
settings where the effective field contains influences of spatial
variations as well.  The proposed method can potentially be used to
approximate macroscopic effects that may arise in different physical
problems such as spin waves.

\bibliographystyle{spmpsci}
\bibliography{dek}

\end{document}